\documentclass[11pt,reqno]{amsart}
 \usepackage[hidelinks]{hyperref}
\usepackage{caption}
 \usepackage{amssymb,amsfonts,amscd,amsbsy, color, amsmath}
\usepackage[mathscr]{eucal}
\usepackage{url}
\usepackage{graphicx}
\usepackage{mathtools}
\usepackage{todonotes}
\usepackage{tabularx}
\usepackage{tabu}
\usepackage{float, color}
\usepackage{placeins}
\usepackage{tikz}

\newtheorem{theorem}{Theorem}[section]
\newtheorem{lemma}[theorem]{Lemma}
\newtheorem{proposition}[theorem]{Proposition}

\theoremstyle{definition}
\newtheorem{remark}[theorem]{Remark}
\numberwithin{equation}{section}

\makeatletter
\def\imod#1{\allowbreak\mkern5mu({\operator@font mod}\,\,#1)}
\makeatother
\allowdisplaybreaks

\makeatletter
\@namedef{subjclassname@2020}{%
  \textup{2020} Mathematics Subject Classification}
\makeatother

\begin{document}

\title[Unimodal sequences and mixed false theta functions]{Unimodal sequences and mixed false theta functions}

\author{Kevin Allen}
\author{Robert Osburn}

\address{School of Mathematics and Statistics, University College Dublin, Belfield, Dublin 4, Ireland}

\email{kevin.allen1@ucdconnect.ie}
\email{robert.osburn@ucd.ie}

\subjclass[2020]{05A15, 11F11, 11F37}
\keywords{Unimodal sequences, Hecke-Appell series, Hecke-type double sums, false theta functions}

\date{\today}

\begin{abstract}
We consider two-parameter generalizations of Hecke-Appell type expansions for the generating functions of unimodal and special unimodal sequences. We then determine their explicit representations which involve mixed false theta functions. These results complement recent striking work of Mortenson and Zwegers on the mixed mock modularity of the generalized $U$-function due to Hikami and Lovejoy. As an application, we demonstrate how to recover classical partial theta function identities which appear in Ramanujan's lost notebook and in work of Warnaar.
\end{abstract}

\maketitle

\section{Introduction}

\subsection{Motivation} A sequence of positive integers is {\it strongly unimodal} if
\begin{equation} \label{strongun}
a_1 < \dotsc < a_r < c > b_1 > \dotsc > b_s 
\end{equation}
with $n = c + \sum_{j=1}^{r} a_j + \sum_{j=1}^{s} b_j$. Here, $c$ is the {\it peak} and $n$ is the {\it weight} of the sequence. The {\it rank} of such a sequence is defined as $s-r$, i.e., the number of terms after $c$ minus the number of terms before $c$. For example, there are six strongly unimodal sequences of weight 5, namely
\begin{equation*}
(5), \, (1, 4), \, (4, 1), \, (1,3,1), \, (2,3), \, (3,2). 
\end{equation*}
The ranks are $0$, $-1$, $1$, $0$, $-1$ and $1$, respectively. Let $u(m,n)$ be the number of such sequences of weight $n$ and rank $m$. A brief reflection reveals the generating function
\begin{equation*} \label{stronggen}
U(x;q) := \sum_{\substack{n \geq 1 \\ m \in \mathbb{Z}}} u(m,n) x^m q^n = \sum_{n \geq 0} (-xq)_n (-x^{-1} q)_n q^{n+1}
\end{equation*}
where $q$ is a non-zero complex number with $|q| < 1$ and
\begin{equation*}
(a)_n = (a;q)_n := \prod_{k=1}^{n} (1-aq^{k-1})
\end{equation*}
is the usual $q$-Pochhammer symbol, valid for $n \in \mathbb{N} \cup \{ \infty \}$. Such sequences not only abound in algebra, combinatorics and geometry  \cite{branden, brenti, stanley}, but have recent intriguing connections to knot theory and modular forms \cite{bopr, masbaum}. In 2015, Hikami and Lovejoy \cite{hl} introduced the generalized $U$-function
\begin{equation} \label{genstrongU}
\begin{aligned} 
U_{t}^{(m)}(x;q) := q^{-t }\sum_{\substack{k_t \geq \dots \geq k_1 \geq 0 \\ k_m \geq 1}} & (-xq)_{k_t-1}(-x^{-1} q)_{k_t-1} q^{k_t} \\
&\times \prod_{i=1}^{t-1} q^{k_i^2} \begin{bmatrix} k_{i+1} - k_i - i + \sum_{j=1}^{i}(2k_j + \chi(m>j)) \\ k_{i+1}-k_i \end{bmatrix}
\end{aligned} 
\end{equation}
where $t$, $m \in \mathbb{Z}$ with $1 \leq m \leq t$,  $\chi(X): = 1$ if $X$ is true and  $\chi(X):= 0$ otherwise  and
\begin{equation*}
\begin{bmatrix} n \\ k \end{bmatrix} := \frac{(q)_n}{(q)_{n-k} (q)_k}
\end{equation*}
is the standard $q$-binomial coefficient. Note that 
\begin{equation*}
U_{1}^{(1)}(x;q) = q^{-1} U(x;q). 
\end{equation*}
The motivation for (\ref{genstrongU}) arises in quantum topology. Let $K$ be a knot and $J_{N}(K;q)$ be the $N$th colored Jones polynomial, normalized to be 1 for the unknot. By computing an explicit formula for the cyclotomic coefficients of the colored Jones polynomial of the left-handed torus knots $T^{*}_{(2,2t+1)}$ \cite[Proposition 3.2]{hl} and comparing with (\ref{genstrongU}), one observes
\begin{equation*} \label{torusgen}
U_t^{(1)}(-q^N; q) = J_N(T^{*}_{(2,2t+1)}; q)
\end{equation*}
and so $U_{t}^{(m)}(x;q)$ can be viewed as ``extracted" from $J_N(T^{*}_{(2,2t+1)}; q)$. In addition, Hikami and Lovejoy proved the Hecke-Appell type expansion \cite[Theorem 5.6]{hl}
\begin{equation} \label{ha}
\begin{aligned}
U_{t}^{(m)}(-x;q) & = -q^{-\frac{t}{2} - \frac{m}{2} + \frac{3}{8}} \frac{(qx)_{\infty} (x^{-1} q)_{\infty}}{(q)_{\infty}^2}  \\ 
& \times \left( \sum_{\substack{r, s \geq 0 \\ r \not\equiv s \;(\text{mod}\;2)}} - \sum_{\substack{r, s < 0 \\ r \not\equiv s \;(\text{mod}\;2)}} \right ) \frac{(-1)^{\frac{r-s-1}{2}} q^{\frac{1}{8} r^2 + \frac{4t+3}{4} rs + \frac{1}{8} s^2 + \frac{1+m+t}{2} r + \frac{1-m+t}{2} s}}{1-xq^{\frac{r+s+1}{2}}}
\end{aligned}
\end{equation}
and stated \cite[page 13]{hl}``$\dotsc$ it is hoped that the Hecke series expansions established in this paper will turn out to be useful for determining modular transformation formulae for $U_{t}^{(m)}(x;q)$." Given that the base case $U_1^{(1)}(x;q)$ is a mixed mock modular form \cite[Theorem 4.1]{hl}, one wonders if the same is true for $U_{t}^{(m)}(x;q)$. Mixed mock modular forms are functions which lie in the tensor space of mock modular forms and modular forms \cite{dmz, losurvey}. In recent striking work \cite{mz}, Mortenson and Zwegers show that this is indeed the case by expressing $U_{t}^{(m)}(x;q)$ in terms of finite sums of Hecke-type double sums
\begin{equation} \label{hecke}
f_{a,b,c}(x,y; q) := \sum_{r, s \in \mathbb{Z}} \text{sg}(r,s) (-1)^{r+s} x^r y^s q^{a \binom{r}{2} + brs + c \binom{s}{2}}
\end{equation}
where $a$, $b$ and $c$ are positive integers,
\begin{equation} \label{sg}
\text{sg}(r,s) := \frac{\text{sg}(r) + \text{sg}(s)}{2}
\end{equation}
and
\begin{equation} \label{sgo}
\text{sg}(r) = 
\begin{cases}
1 & \text{if $r \geq 0$}, \\
-1 & \text{if $r < 0$.}
\end{cases}
\end{equation}
Precisely, they prove for $t \geq 2$ and $1 \leq m \leq t$ \cite[Theorem 1.7, Corollary 5.3]{mz}
\begin{equation} \label{umixed}
\begin{aligned}
(1-x) U_{t-1}^{(m)}(-x;q) & = \frac{q^{-m+1-t}}{(q)_{\infty}^3} \sum_{k=0}^{2t-1} (-1)^k q^{\binom{k+1}{2}} \\
& \qquad \times \left( f_{1,4t-1, 1}(q^{k+m+t}, q^{k-t-m+1};q) - q^m f_{1,4t-1,1}(q^{k-t+m+1}, q^{k-m+t}; q) \right)  \\
& \qquad  \times f_{1,2t,2t(2t-1)}(x^{-1} q^{1+k}, -q^{(2t-1)(k+t) +t}; q). \\
\end{aligned}
\end{equation}
As discussed below, one can show that the expression within the parenthesis in (\ref{umixed}) is (up to an appropriate power of $q$) a modular form while the remaining double sum is a mixed mock modular form. Thus, $U_{t}^{(m)}(x;q)$ is a mixed mock modular form.

\subsection{Statement of Results}
A sequence of positive integers is {\it unimodal} if each $<$ is replaced with $\leq$ and we write $\overline{c}$ for the distinguished peak in (\ref{strongun}). For example, there are twelve unimodal sequences of weight 4, namely
\begin{equation*}
\begin{aligned}
 (\overline{4}), \, (1, \overline{3}), \, (\overline{3}, 1), &\, (1, \overline{2}, 1), \, (\overline{2}, 2), \, (2, \overline{2}), \\
(1,1,\overline{2}), \, (\overline{2},1, 1), \, (\overline{1}, 1, 1, 1), & \, (1, \overline{1}, 1, 1), \, (1, 1, \overline{1}, 1), \, (1,1,1,\overline{1}).
\end{aligned}
\end{equation*}
These sequences have numerous other guises \cite[Section 3]{andrewsIV} and appear in a wide variety of areas \cite{murty}. The rank of such a sequence is again $s-r$. Inspired by (\ref{ha}) and (\ref{umixed}), the purpose of this paper is to first consider two-parameter generalizations of Hecke-Appell type expansions for four of the five generating functions of unimodal and special unimodal sequences which appear in \cite{bl, kl1, kl2}. We also briefly discuss a two-parameter generalization of the Hecke-Appell type expansion for the fifth case \cite{kl2} (see Section 4). To our knowledge, this covers all known cases of unimodal sequences whose generating function has such an expansion. We then demonstrate that the techniques in \cite{mz} are robust enough to find explicit representations for {\it all} of these generalizations. These representations involve mixed false theta functions, i.e., expressions of the form $\sum h_v g_v$ where $h_v$ is a modular form and $g_v$ is a false theta function \cite{rogers}. These new occurrences of mixed modularity nicely complement (\ref{umixed}) and hint at a general underlying structure for Hecke-Appell type expansions with such properties. This structure will be the subject of forthcoming work. Let $D := b^2 - ac$. The key is to express each of these two-parameter generalizations in terms of finite sums of (\ref{hecke}) and then apply \cite[Corollary 4.2]{mz} if $D>0$ or \cite[Theorem 1.4]{mortfalse} if $D< 0$, both of which we recall in Section 2. For example, one uses \cite[Corollary 4.2]{mz} to deduce the mixed mock modularity of (\ref{umixed}). As an application, we demonstrate how the base cases of our results recover classical partial theta function identities which appear in Ramanujan's lost notebook and in work of Warnaar.

Let
\begin{equation} \label{theta}
\Theta(x;q) := (x)_{\infty} (q/x)_{\infty} (q)_{\infty} = \sum_{n \in \mathbb{Z}} (-1)^n q^{\binom{n}{2}} x^n.
\end{equation}
For the first case, let ${\bf u}(m,n)$ denote the number of unimodal sequences of weight $n$ and rank $m$ and consider its generating function \cite[Eq. (2.2)]{kl1}
\begin{equation*}
\mathcal{U}(x;q) := \sum_{\substack{n \geq 1 \\ m \in \mathbb{Z}}} {\bf u}(m,n) x^m q^n = \sum_{n \geq 0} \frac{q^n}{(xq)_n (q/x)_n}
\end{equation*}
which satisfies the ``marvelous" partial theta function identity \cite{andrewslost, andrewsI}, \cite[Entry 6.3.2]{abII}
\begin{equation} \label{pt1}
\mathcal{U}(-x;q) = (1+x) \sum_{n \geq 0} x^{3n} q^{\frac{n(3n+1)}{2}} (1 - x^2 q^{2n+1}) - \frac{(1+x) (q)_{\infty}}{\Theta(-x;q)} \sum_{n \geq 0} (-1)^n x^{2n+1} q^{\frac{n(n+1)}{2}}
\end{equation}
for $x\neq 0$ and the Hecke-Appell type expansion \cite[Eq. (2.5)]{kl1}
\begin{equation} \label{base1}
\mathcal{U}(x;q) = \frac{(1-x)}{(q)_{\infty}^2} \left( \sum_{r,s \geq 0} - \sum_{r,s < 0} \right) \frac{(-1)^{r+s} q^{\frac{r^2}{2} + 2rs + \frac{s^2}{2} + \frac{3}{2}r + \frac{1}{2}s}}{1-xq^r}.
\end{equation}
For $t$, $m \in \mathbb{Z}$ with $t \geq 1$, $-t \leq m \leq 3t-2$ and $t \equiv m \pmod{2}$, consider the generalization
\begin{equation} \label{gen1}
g_{t,m}(x;q) := \left( \sum_{r,s \geq 0} - \sum_{r,s < 0} \right) \frac{(-1)^{r+s} q^{\frac{r^2}{2} + 2trs + \frac{s^2}{2} + \frac{t+1+m}{2} r + \frac{t+1-m}{2}s}}{1-xq^r}
\end{equation}
and
\begin{equation} \label{realgen1}
\mathcal{U}_{t}^{(m)}(x;q) := \frac{(1-x)}{(q)_{\infty}^2} g_{t,m}(x;q).
\end{equation}
By (\ref{base1})--(\ref{realgen1}), $\mathcal{U}_1^{(1)}(x;q) = \mathcal{U}(x;q)$. Following \cite{hm}, we use the term ``generic" to mean that the parameters do not cause poles in the Appell-Lerch series (\ref{alsums}) or in the quotients of theta functions which occur after applying (\ref{mzftom}) to the Hecke-type double sums. Our first result shows that $\mathcal{U}_{t}^{(m)}(x;q)$ is a mixed false theta function.

\begin{theorem} \label{main1} Let $t$, $m \in \mathbb{Z}$ with $t \geq 1$, $-t \leq m \leq 3t-2$ and $t \equiv m \pmod{2}$. For generic $x$, we have
\begin{equation} \label{mf1}
\begin{aligned}
\mathcal{U}_{t}^{(m)}(x;q) & = \frac{(1-x)}{\Theta(x;q)}\frac{q^{1-3t^2 - \frac{t-m}{2} + tm}}{(q)_{\infty}^2} \sum_{k=0}^{4t^2 - 2} (-1)^{k+1} q^{\binom{k+1}{2} + k} f_{1,2t,1}(q^{2-4t^2 + \frac{t+m}{2} + k}, q^{1 + \frac{t-m}{2}}; q) \\
& \qquad \qquad \qquad \qquad \qquad \qquad  \times f_{1,4t^2 - 1, 4t^2 (4t^2 - 1)}(x^{-1} q^{k+1}, -q^{4t^2 k - t^2 + tm - \frac{t-m}{2} + 8t^4}; q).
\end{aligned}
\end{equation}
\end{theorem}

For the second case, consider unimodal sequences with a double peak, i.e., sequences of the form
\begin{equation*}
a_1 \leq \dotsc \leq a_r \leq \overline{c} \, \overline{c} \geq b_1 \geq \dotsc \geq b_s 
\end{equation*}
with weight $n=2c + \sum_{i=1}^{r} a_i + \sum_{i=1}^{s} b_i$. For example, there are eleven such sequences of weight 6, namely
\begin{equation*}
\begin{aligned}
(\overline{3}, \overline{3}), \, (\overline{2}, \overline{2}, 2), \, (2, \overline{2}, \overline{2}), \, (\overline{2}, \overline{2}, 1, 1), \, & (1, \overline{2}, \overline{2}, 1), \, (1, 1, \overline{2}, \overline{2}), \\
(\overline{1}, \overline{1}, 1, 1, 1, 1), \, (1, \overline{1}, \overline{1}, 1, 1, 1), \, (1, 1, \overline{1}, \overline{1}, 1, 1), & \, (1, 1, 1, \overline{1}, \overline{1}, 1), \, (1, 1, 1, 1, \overline{1}, \overline{1}).
\end{aligned}
\end{equation*}
The rank of such a unimodal sequence is $s-r$ where we assume that the empty sequence has rank 0. Let $W(m,n)$ denote the number of such sequences of weight $n$ and rank $m$ and consider its generating function \cite[Eq. (2.1)]{kl2}
\begin{equation*}
W(x;q) := \sum_{\substack{n \geq 0 \\ m \in \mathbb{Z}}} W(m,n) x^m q^n  = \sum_{n \geq 0} \frac{q^{2n}}{(xq)_n (q/x)_n}
\end{equation*}
which satisfies the partial theta function identity \cite[p. 379]{warnaar}
\begin{equation} \label{pt2}
\begin{aligned}
W(x;q) & = (1-x) \Biggl( 1+x + (1+x^2) \sum_{n \geq 1} (-1)^n x^{3n-2} q^{\frac{n(3n-1)}{2}} (1 + xq^n) \\
& \qquad \qquad \qquad \qquad \qquad \qquad \qquad + \frac{x^2 + (1+x^2) \sum_{n \geq 1} (-1)^n x^{2n} q^{\binom{n+1}{2}}}{(x)_{\infty} (q/x)_{\infty}} \Biggr)
\end{aligned}
\end{equation}
for $x \neq 0$ and the Hecke-Appell type expansion \cite[Eq. (2.3)]{kl2}
\begin{equation} \label{base2}
W(x;q) =  \frac{(1-x)}{(q)_{\infty}^2} \left( \sum_{r,s \geq 0} - \sum_{r,s < 0} \right) \frac{(-1)^{r+s} q^{\frac{r^2}{2} + 2rs + \frac{s^2}{2} + \frac{r}{2} + \frac{s}{2}} (1+q^{2r})}{1-xq^r}  - \frac{1}{(xq)_{\infty} (q/x)_{\infty}}.
\end{equation}
For $t$, $m \in \mathbb{Z}$ with $t \geq 1$, $1-t \leq m \leq t$, consider the generalization
\begin{equation} \label{gen2}
h_{t,m}(x;q) := \left( \sum_{r,s \geq 0} - \sum_{r,s < 0} \right) \frac{(-1)^{r+s} q^{\binom{r+1}{2} + 2trs + \binom{s+1}{2} + (t-m)r} (1+q^{2mr})}{1-xq^r}
\end{equation}
and  
\begin{equation} \label{realgen2}
W_{t}^{(m)}(x;q) := \frac{(1-x)}{(q)_{\infty}^2} h_{t,m}(x;q)  - \frac{1}{(xq)_{\infty} (q/x)_{\infty}}.
\end{equation}
By (\ref{base2})--(\ref{realgen2}), $W_{1}^{(1)}(x;q) = W(x;q)$. Our second result demonstrates that $W_{t}^{(m)}(x;q)$ is the sum of a mixed false theta function and a modular form.

\begin{theorem} \label{main2} Let $t$, $m \in \mathbb{Z}$ with $t \geq 1$, $1-t \leq m \leq t$. For generic $x$, we have
\begin{equation} \label{mf2}
\begin{aligned}
W_{t}^{(m)}(x;q) & = \frac{(1-x)q^{1-m-2t^2}}{(q)_{\infty}^2 \Theta(x;q)} \sum_{k=0}^{4t^2-2} (-1)^{k+1} q^{\binom{k+1}{2} +k} f_{1,2t,1}(q^{t-m+2-4t^2+k}, q; q) \\
& \qquad\qquad \qquad \qquad \quad \times \Biggl( f_{1,4t^2-1, 4t^2(4t^2-1)}(x^{-1} q^{k+1}, -q^{8t^4 - m + 4t^2 k}; q) \\
&  \qquad \qquad \qquad \qquad \qquad \quad + f_{1, 4t^2 - 1, 4t^2 (4t^2 - 1)}(x q^{k+1}, -q^{8t^4 - m + 4t^2k}; q) \Biggr) \\
& - \frac{1}{(xq)_{\infty} (q/x)_{\infty}}.
\end{aligned}
\end{equation}
\end{theorem}

For the third case, recall that an overpartition is a partition in which the first occurrence of a part may be overlined. Consider unimodal sequences where $c$ is odd, $\sum a_i$ is a partition without repeated even parts and $\sum b_i$ is an overpartition into odd parts whose largest part is not $\overline{c}$. For example, there are twelve such sequences of weight 5, namely
\begin{equation*}
\begin{aligned}
(\overline{5}), \, (1, \overline{3}, 1), \, (1, 1, \overline{3}), \, (\overline{3}, 1, 1), \, (\overline{3}, \overline{1}, 1), \, (1, \overline{3}, \overline{1}), & \, (2, \overline{3}), \\
(1, 1, 1, 1, \overline{1}), \, (1, 1, 1, \overline{1}, 1), \, (1, 1, \overline{1}, 1, 1), \, (1, \overline{1}, 1, 1, 1), & \, (\overline{1}, 1, 1, 1, 1).
\end{aligned}
\end{equation*}
The rank of such a sequence is the number of odd non-overlined parts in $\sum b_i$ minus the number of odd parts in $\sum a_i$ where the empty sequence is assumed to have rank 0. Let $\mathcal{V}(m,n)$ denote the number of such sequences of weight $n$ and rank $m$ and consider its generating function \cite[Eq. (4.1)]{kl2}
\begin{equation*}
\mathcal{V}(x;q) :=  \sum_{\substack{n \geq 0 \\ m \in \mathbb{Z}}} \mathcal{V}(m,n) x^m q^n = \sum_{n \geq 0} \frac{(-q)_{2n} q^{2n+1}}{(xq;q^2)_{n+1} (q/x;q^2)_{n+1}}
\end{equation*}
which satisfies the partial theta function identity \cite[Entry 6.3.7]{abII}
\begin{equation} \label{pt3}
\left(1+\frac{1}{x} \right) \mathcal{V}(x;q) = - \sum_{n \geq 0} (-x)^n q^{n(n+1)} + \frac{\Theta(-q;q^4)}{\Theta(xq;q^2)} \sum_{n \geq 0} (-x)^n q^{\frac{n(n+1)}{2}}
\end{equation}
for $x \neq 0$ and the Hecke-Appell type expansion \cite[Eq. (4.3)]{kl2}
\begin{equation} \label{base3}
\mathcal{V}(x;q) =  \frac{1}{(q)_{\infty} (q^2; q^2)_{\infty}} \left( \sum_{r, s \geq 0} - \sum_{r,s < 0} \right) \frac{(-1)^{r+s} q^{r^2 + 2rs + \frac{s^2}{2} + 3r + \frac{3s}{2} + 1}}{(1+q^{2r+1}) (1-xq^{2r+1})}.
\end{equation}
For $t$, $m \in \mathbb{Z}$ where $t \geq 1$, consider the generalization
\begin{equation*} \label{gen4}
\overline{k}_{t,m}(x;q) = \frac{xq}{1+x} \left( \frac{1}{x} k_{t,m}(-1;q) + k_{t,m}(x;q) \right)
\end{equation*}
where
\begin{equation} \label{gen41}
k_{t,m}(x;q) := \left( \sum_{r, s \geq 0} - \sum_{r, s < 0} \right) \frac{(-1)^{r+s} q^{2\binom{r}{2} + 2trs + \binom{s}{2} + 2(t+1)r + 2ms}}{1-xq^{2r+1}}
 \end{equation}
and
\begin{equation} \label{realgen4}
\mathcal{V}_{t}^{(m)}(x;q) : =  \frac{1}{(q)_{\infty} (q^2; q^2)_{\infty}}  \overline{k}_{t,m}(x;q).
\end{equation}
By (\ref{base3})--(\ref{realgen4}), $\mathcal{V}_{1}^{(1)}(x;q) = \mathcal{V}(x;q)$. Our third result establishes that $\mathcal{V}_{t}^{(m)}(x;q)$ is a mixed false theta function.

\begin{theorem} \label{main3} Let $t$, $m \in \mathbb{Z}$ where $t \geq 1$. For generic $x$, we have
\begin{equation} \label{mf3}
\begin{aligned}
\mathcal{V}_{t}^{(m)}(x;q) & = \frac{q^{-2t^2 + 3t - 4tm}}{(1+x) (q)_{\infty} (q^2; q^2)_{\infty}} \sum_{k=0}^{2t^2 - 2} (-1)^{k} q^{k^2 + 3k} f_{2, 2t, 1}(q^{2t+2 - 4t^2 + 2k}, q^{2m}; q) \\
& \qquad \times \Biggl( \frac{1}{\Theta(-q;q^2)} f_{1, 2t^2 - 1, 2t^2 (2t^2 - 1)}(-q^{2k+1}, -q^{4t^2 k - 2 + 3t - 4tm + 4t^4} ; q^2) \\
& \qquad \qquad - \frac{1}{\Theta(qx;q^2)}  f_{1, 2t^2 - 1, 2t^2 (2t^2 - 1)}(x^{-1} q^{2k+1}, -q^{4t^2 k - 2 + 3t - 4tm + 4t^4}; q^2) \Biggr).
\end{aligned}
\end{equation}
\end{theorem}

For the final case, consider odd unimodal sequences, i.e., unimodal sequences where the parts $a_i$, $b_j$ and $c$ are odd positive integers. For example, there are six such sequences of weight 4, namely
\begin{equation*}
(1, \overline{3}), \, (\overline{3}, 1), \, (\overline{1}, 1, 1, 1), \, (1, \overline{1}, 1, 1), \, (1, 1, \overline{1}, 1), \, (1, 1, 1, \overline{1}).
\end{equation*}
Again, the rank is $s-r$. Let $\text{ou}(m,n)$ denote the number of odd unimodal sequences of weight $n$ and rank $m$ and consider its generating function  \cite[Eq. (1.5)]{bl}
\begin{equation*}
\mathcal{O}(x;q) := \sum_{\substack{n \geq 1 \\ m \in \mathbb{Z}}} \text{ou}(m,n) x^m q^n = \sum_{n \geq 0} \frac{q^{2n+1}}{(xq;q^2)_{n+1} (q/x;q^2)_{n+1}}
\end{equation*}
which satisfies the partial theta function identity \cite{andrewsI}, \cite[Entry 6.3.4]{abII}
\begin{equation} \label{pt4}
\mathcal{O}(-x;q) = \sum_{n \geq 0} x^{3n+1} q^{3n^2 + 2n} (1 - xq^{2n+1}) - \frac{(q^2;q^2)_{\infty}}{\Theta(-xq;q^2)} \sum_{n \geq 0} (-1)^n x^{2n+1} q^{n(n+1)} 
\end{equation}
for $x \neq 0$ and the Hecke-Appell type expansion \cite[Eq. (1.7)]{bl}
\begin{equation} \label{base5}
\mathcal{O}(x;q) = \frac{q}{(q^2; q^2)_{\infty}^2}  \left( \sum_{r,s \geq 0} - \sum_{r,s < 0} \right) \frac{(-1)^{r+s} q^{r^2 + 4rs + s^2 + 3r + 3s}}{1-xq^{2r+1}}.
\end{equation}
For $t$, $m \in \mathbb{Z}$ with $t \geq 1$, $1-t \leq m \leq t$, consider the generalization
\begin{equation} \label{gen5}
p_{t,m}(x;q) := \left( \sum_{r,s \geq 0} - \sum_{r,s < 0} \right) \frac{(-1)^{r+s} q^{\binom{r}{2} + 2trs + \binom{s}{2} + (t+m)r + (t+m)s}}{1-xq^r}
\end{equation}
and
\begin{equation} \label{realgen5}
\mathcal{O}_{t}^{(m)}(x;q) := \frac{q}{(q^2; q^2)_{\infty}^2} p_{t,m}(qx;q^2).
\end{equation}
By (\ref{base5})--(\ref{realgen5}), $\mathcal{O}_{1}^{(1)}(x;q) = \mathcal{O}(x;q)$. Our last result exhibits that $\mathcal{O}_{t}^{(m)}(x;q)$ is a mixed false theta function.

\begin{theorem} \label{main4} Let $t$, $m \in \mathbb{Z}$ with $t \geq 1$, $1-t \leq m \leq t$. For generic $x$, we have
\begin{equation} \label{mf4}
\begin{aligned}
\mathcal{O}_{t}^{(m)}(x;q) & = \frac{q^{3- 8t^2 - 2(m-1)(2t-1)}}{\Theta(xq;q^2) (q^2; q^2)_{\infty}^2} \sum_{k=0}^{4t^2-2} (-1)^{k+1} q^{k^2 + 3k} f_{1,2t,1}(q^{2t+2m+2-8t^2 + 2k}, q^{2t + 2m}; q^2) \\
& \qquad \qquad \qquad \qquad \times f_{1,4t^2 - 1, 4t^2 (4t^2 - 1)}(x^{-1} q^{2k+1}, - q^{16t^4 - 4t^2 - 2(m-1)(2t-1) + 8t^2k}; q^2).
\end{aligned}
\end{equation}
\end{theorem}

\begin{remark}
A two-parameter generalization in the case of the Hecke-Appell type expansion for odd strongly unimodal sequences \cite[Eq. (1.11)]{bl} can be found in \cite[Theorem 1.6, Corollary 1.7]{boro} and is a mixed mock modular form. 
\end{remark}

The paper is organized as follows. In Section 2, we first recall the key results from \cite{mortfalse, mz}. We then provide the necessary background on identities for (\ref{sg}) and (\ref{theta}) and prove alternative forms for (\ref{gen2}) and (\ref{gen41}). In Section 3, we prove Theorems \ref{main1}--\ref{main4}. Observe that the starting point of \cite{mz} is (\ref{ha}) whereas in this paper one needs to initially construct the two-parameter generalizations (\ref{gen1}), (\ref{gen2}), (\ref{gen41}) and (\ref{gen5}). Also, crucial steps in the proofs of Theorems \ref{main1}--\ref{main4} include the discovery of the appropriate sums (\ref{sum}), (\ref{sum1}), (\ref{sum3}) and (\ref{sum4}) and a novel argument in proving (\ref{keylim}), (\ref{newkeylim1}) and (\ref{keylim1}) (cf. \cite[(2.2)]{mz}). In Section 4, we demonstrate how the $t=m=1$ cases of Theorems \ref{main1}--\ref{main4} recover (\ref{pt1}), (\ref{pt2}), (\ref{pt3}) and (\ref{pt4}). As the reader will see, these verifications require some dexterity. In Section 5, we make some concluding remarks and discuss future directions. 

\section{Preliminaries}

We begin with two important results, the first converts Hecke-type double sums (\ref{hecke}) into Appell-Lerch series 
\begin{equation} \label{alsums}
m(x,q,z) := \frac{1}{\Theta(z;q)} \sum_{r \in \mathbb{Z}} \frac{(-1)^r q^{\binom{r}{2}} z^r}{1-q^{r-1} xz}
\end{equation}
while the second expresses (\ref{hecke}) in terms of mixed false theta functions. Here, $x$ and $z$ are non-zero complex numbers with neither $z$ nor $xz$ an integral power of $q$. Note that specializations of (\ref{alsums}) give mock theta functions \cite{zwegers}.

\begin{theorem}[\cite{mz}, Corollary 4.2] \label{mzmock} For $D := b^2 - ac > 0$ and generic $x$ and $y$, we have 
\begin{equation} \label{mzftom}
f_{a,b,c}(x,y;q) = g_{a,b,c}(x,y, -1, -1; q) + \frac{1}{\Theta(-1; q^{aD}) \Theta(-1;q^{cD})} \vartheta_{a,b,c}(x, y; q)
\end{equation}
where
\begin{equation*}
\begin{aligned}
g_{a,b,c}(x, y, z_1, z_0; q) & := \sum_{i=0}^{a-1} (-y)^i q^{c \binom{i}{2}} \Theta(q^{bi}x; q^a) m(-q^{a \binom{b+1}{2} - c \binom{a+1}{2} - iD} \frac{(-y)^a}{(-x)^b}, z_0; q^{aD}) \\
& + \sum_{i=0}^{c-1} (-x)^i q^{a \binom{i}{2}} \Theta(q^{bi}y; q^c)  m(-q^{c \binom{b+1}{2} - a \binom{c+1}{2} - iD} \frac{(-x)^c}{(-y)^b}, z_1; q^{cD}), 
\end{aligned}
\end{equation*}
\begin{equation*}
\begin{aligned}
\vartheta_{a,b,c}(x, y; q) & := \sum_{d^{*}=0}^{b-1} \sum_{e^{*}=0}^{b-1} q^{a \binom{d- c/2}{2} + b(d- c/2)(e+ a/2) + c\binom{e+ a/2}{2}} (-x)^{d - c/2} (-y)^{e+ a/2} \\
& \times \sum_{f=0}^{b-1}q^{ab^2 \binom{f}{2} + \bigl( a(bd + b^2 + ce) - ac(b+1)/2 \bigr) f} (-y)^{af} \Theta(-q^{c(ad + be + a(b-1)/2 + abf)} (-x)^c ; q^{cb^2}) \\
& \times \Theta(-q^{a ((d+ b(b+1)/2 + bf)D + c(a-b)/2)} (-x)^{-ac} (-y)^{ab}; q^{ab^2 D}) \\
& \times \frac{(q^{bD}; q^{bD})_{\infty}^3 \Theta(q^{D(d+e) + ac - b(a+c)/2} (-x)^{b-c} (-y)^{b-a} ; q^{bD})}{\Theta(q^{De + a(c-b)/2} (-x)^b (-y)^{-a} ; q^{bD}) \, \Theta(q^{Dd + c(a-b)/2} (-y)^b (-x)^{-c}; q^{bD})},
\end{aligned}
\end{equation*}
$d:= d^{*} + \{ c/2 \}$ and $e := e^{*} + \{ a/2 \}$ with $0 \leq \{ \alpha \} < 1$ denoting the fractional part of $\alpha$.
\end{theorem}

\begin{theorem}[\cite{mortfalse}, Theorem 1.4] \label{mortf} For $D := b^2 - ac < 0$, we have
\begin{equation} \label{ftofalse}
\begin{aligned}
f_{a,b,c}(x,y;q) & = \frac{1}{2} \Biggl( \sum_{i=0}^{a-1} (-y)^i q^{c \binom{i}{2}} \Theta(q^{bi}x; q^a) \sum_{r \in \mathbb{Z}} \text{sg}(r)  \left( q^{a \binom{b+1}{2} - c \binom{a+1}{2} - iD} \frac{(-y)^a}{(-x)^b} \right)^r q^{-aD \binom{r+1}{2}} \\ 
& \qquad \quad + \sum_{i=0}^{c-1} (-x)^i q^{a \binom{i}{2}} \Theta(q^{bi} y; q^c)  \sum_{r \in \mathbb{Z}} \text{sg}(r)  \left( q^{c \binom{b+1}{2} - a \binom{c+1}{2} - iD} \frac{(-x)^c}{(-y)^b} \right)^r q^{-cD \binom{r+1}{2}} \Biggr).
\end{aligned}
\end{equation}
\end{theorem}

We continue with a result concerning identities satisfied by $\text{sg}(r,s)$. We omit the proof. Let
\begin{equation*}
\delta(r) := 
\begin{cases}
1 & \text{if $r=0$}, \\
0 & \text{$\text{otherwise.}$}
\end{cases}
\end{equation*}
\begin{lemma} For $r$, $s$, $t \in \mathbb{Z}$ with $t \geq 1$, we have
\begin{equation} \label{prop0}
\textup{sg}(-r,-s-1) = -\textup{sg}(r,s) + \delta(r)
\end{equation}
and
\begin{equation} \label{prop1}
\textup{sg}(r-1,s+2t)=\textup{sg}(r,s) - \delta(r) + \sum_{i=1}^{2t} \delta(s+i).
\end{equation}
\end{lemma}
We now recall the theta function identities
\begin{equation} \label{theta2}
\Theta(q^n; q) = 0,
\end{equation}
\begin{equation} \label{theta1}
\Theta(q^n x; q) = (-1)^n q^{-\binom{n}{2}} x^{-n} \Theta(x;q)
\end{equation}
and
\begin{equation} \label{theta3}
\sum_{k\in\mathbb{Z}}\frac{(-1)^kq^{\frac{1}{2}k^2+(n+\frac{1}{2})k}}{1-xq^k} = \frac{(q)^3_{\infty}}{x^n\Theta(x;q)}
\end{equation}
where $n \in \mathbb{Z}$. Next, we turn to providing alternative expressions for (\ref{gen2}) and (\ref{gen41}) which will be beneficial in the proofs of Theorems \ref{main2} and \ref{main3}, respectively. Let 
\begin{equation} \label{piece}
\mathcal{H}_{t}^{(m)}(x;q) := \sum_{r,s \in \mathbb{Z}} \text{sg}(r,s) (-1)^{r+s} \frac{q^{\binom{r+1}{2} + 2trs + \binom{s+1}{2} + (t-m)r}}{1-xq^r}
\end{equation}
and
\begin{equation} \label{kappa}
\kappa_{t,m}(x;q) := \sum_{r,s \in \mathbb{Z}} \text{sg}(r,s) (-1)^{r+s} \frac{q^{2\binom{r}{2} + 2trs + \binom{s}{2} + 2tr + 2ms}}{1-xq^{2r+1}}. 
\end{equation}
\begin{proposition} \label{help} We have
\begin{equation} \label{seconddecompose}
h_{t,m}(x;q) = \mathcal{H}_{t}^{(m)}(x;q) - x^{-1} \mathcal{H}_{t}^{(m)}(x^{-1};q)
\end{equation}
and
\begin{equation} \label{split}
k_{t,m}(x;q) = x^{-1} q^{-1} \left( \kappa_{t,m}(x;q) - f_{2,2t,1}(q^{2t}, q^{2m}; q) \right).
\end{equation}
\end{proposition}

\begin{proof}
We first let $(r, s) \to (-r-1, -s-1)$ in $\mathcal{H}_{t}^{(m)}(x^{-1};q)$ and simplify to obtain
\begin{equation} \label{hsim}
-xq^{1+m+t} \sum_{r,s \in \mathbb{Z}} \text{sg}(-r-1, -s-1) (-1)^{r+s} \frac{q^{\binom{r+1}{2} + 2trs + \binom{s+1}{2} + (m+t+1)r + 2ts}}{1-xq^{r+1}}.
\end{equation}
Next, applying $r \to r-1$ to (\ref{hsim}), then using (\ref{theta}) and (\ref{prop0}) yields
\begin{equation*}
 -x \sum_{r,s \in \mathbb{Z}} \text{sg}(r,s) (-1)^{r+s} \frac{q^{\binom{r+1}{2} + 2trs + \binom{s+1}{2} + (m+t)r}}{1-xq^r}
 \end{equation*}
and thus (\ref{seconddecompose}) follows. Using (\ref{hecke}), one observes that
\begin{equation*}
f_{2,2t,1}(q^{2t}, q^{2m}; q) = \kappa_{t,m}(x;q) - xq k_{t,m}(x;q)
\end{equation*} 
and so (\ref{split}) follows.
\end{proof}

\section{Proof of Theorems \ref{main1}--\ref{main4}}

The method of proof is as follows \cite{mz}. First, we derive functional equations for each of (\ref{gen1}), (\ref{gen2}), (\ref{gen41}), (\ref{gen5}),
\begin{equation} \label{t1}
\hat{g}_{t,m}(x;q) := \Theta(x;q) g_{t,m}(x;q),
\end{equation}
\begin{equation} \label{t2}
\hat{\mathcal{H}}_{t}^{(m)}(x;q) := \Theta(x;q) \mathcal{H}_{t}^{(m)}(x;q),
\end{equation}
\begin{equation} \label{t3}
\hat{\kappa}_{t,m}(x;q) := \Theta(qx;q^2) {\kappa}_{t,m}(x;q)
\end{equation}
and
\begin{equation} \label{t4}
\hat{p}_{t,m}(x;q) := \Theta(x;q) p_{t,m}(x;q).
\end{equation}
Suitable care is required in constructing the sums (\ref{sum}), (\ref{sum1}), (\ref{sum3}) and (\ref{sum4}) which favorably decompose in order to obtain these functional equations. We then express each of (\ref{t1})--(\ref{t4}) as a Laurent series in $x \in \mathbb{C} \setminus \{ 0\}$ and use the functional equations to find an explicit formula for the coefficients in the Laurent series expansion. After some sitzfleisch, these calculations eventually yield the right-hand sides of (\ref{mf1}), (\ref{mf2}), (\ref{mf3}) and (\ref{mf4}). For the first case, we begin with the following result. 

\begin{proposition} For $t \in \mathbb{N}$ and $m \in \mathbb{Z}$ with $t \equiv m \pmod{2}$, we have
\begin{equation} \label{r1}
\begin{aligned}
g_{t,m}(qx;q) &= -x^{1-4t^2}q^{\frac{t-m-2t^2 -2tm}{2}} g_{t,m}(x;q) \\
& - x^{1-4t^2 - \frac{t+m}{2}} q^{\frac{t-m-2t^2-2tm}{2}} \frac{(q)_{\infty}^{3}}{\Theta(x;q)}\sum_{i=1}^{2t}(-1)^iq^{\frac{i^2}{2}-\frac{1+t-m}{2}i}x^{2ti}\\
&- x^{1-4t^2} q^{1-4t^2} \sum_{k=0}^{4t^2-2}x^kq^{k} f_{1,2t,1}(q^{2-4t^2 + \frac{t+m}{2} + k}, q^{1 + \frac{t-m}{2}}; q)
\end{aligned}
\end{equation}
and 
\begin{equation} \label{r2}
\begin{aligned}
\hat{g}_{t,m}(qx;q) &= x^{-4t^2}q^{\frac{t-m-2t^2-2tm}{2}} \hat{g}_{t,m}(x;q) \\
& + x^{-4t^2-\frac{t+m}{2}}q^{\frac{t-m-2t^2-2tm}{2}}(q)^3_{\infty}\sum_{i=1}^{2t}(-1)^iq^{\frac{i^2}{2}-\frac{1+t-m}{2}i}x^{2ti}\\
& + x^{-4t^2} q^{1-4t^2}\Theta(x;q)\sum_{k=0}^{4t^2-2}x^{k} q^{k} f_{1,2t,1}(q^{2-4t^2 + \frac{t+m}{2} + k}, q^{1 + \frac{t-m}{2}}; q).
 \end{aligned}
\end{equation}
\end{proposition}

\begin{proof}
The idea is to compute the sum
\begin{equation} \label{sum}
x^{4t^2-1}q^{\frac{2t^2-t+2tm+m}{2}}\sum_{r,s \in \mathbb{Z}} \text{sg}(r,s)(-1)^{r+s}q^{\frac{r^2}{2}+2trs+\frac{s^2}{2}+\frac{1+t+m}{2}r+\frac{1+t-m}{2}s}\frac{1-x^{1-4t^2}q^{(r+1)(1-4t^2)}}{1-xq^{r+1}}
\end{equation}
in two ways. Expanding the numerator yields
\begin{equation} \label{expand}
x^{4t^2-1}q^{\frac{2t^2-t+2tm+m}{2}} g_{t,m}(qx;q) - \sum_{r,s \in \mathbb{Z}}\text{sg}(r,s)(-1)^{r+s}\frac{q^{\frac{r^2}{2}+2trs+\frac{s^2}{2}+\frac{1+t+m}{2}r+\frac{1+t-m}{2}s+ 1- 3t^2 - \frac{t}{2} + tm + \frac{m}{2}}}{1-xq^{r+1}}.
\end{equation}
Taking $(r,s) \to (r-1, s+2t)$ in the second sum in (\ref{expand}) and using (\ref{prop1}), (\ref{theta2}) and (\ref{theta3}) leads to
\begin{equation} \label{oneside}
\begin{aligned}
& -\sum_{r,s \in \mathbb{Z}} \text{sg}(r-1,s+2t)(-1)^{r+s}\frac{q^{\frac{r^2}{2}+2trs+\frac{s^2}{2}+\frac{1+t+m}{2}r+\frac{1+t-m}{2}s}}{1-xq^r} \\
& = - \sum_{r,s \in \mathbb{Z}} \text{sg}(r,s)(-1)^{r+s}\frac{q^{\frac{r^2}{2}+2trs+\frac{s^2}{2}+\frac{1+t+m}{2}r+\frac{1+t-m}{2}s}}{1-xq^r} + \frac{1}{1-x}\sum_{s\in\mathbb{Z}}(-1)^sq^{\frac{s^2+(1+t-m)s}{2}} \\
& - \sum_{i=1}^{2t}\sum_{r \in \mathbb{Z}}(-1)^{r-i}\frac{q^{\frac{r^2}{2}-2tri+\frac{i^2}{2}+\frac{1+t+m}{2}r-\frac{1+t-m}{2}i}}{1-xq^r} \\
& = - g_{t,m}(x;q) - \frac{x^{-\frac{t+m}{2}}(q)_{\infty}^{3}}{\Theta(x;q)}\sum_{i=1}^{2t}(-1)^iq^{\frac{i^2}{2}-\frac{1+t-m}{2}i}x^{2ti}. 
\end{aligned}
\end{equation}
Alternatively, we use
\begin{equation} \label{geo}
\frac{1-x^{1-4t^2}q^{(r+1)(1-4t^2)}}{1-xq^{r+1}} = -x^{1-4t^2}q^{(r+1)(1-4t^2)}\sum_{k=0}^{4t^2-2}x^kq^{k(r+1)}
\end{equation}
to express (\ref{sum}) as
\begin{equation} \label{otherside}
\begin{aligned}
&  -q^{\frac{-6t^2-t+2tm+m}{2}+1}\sum_{k=0}^{4t^2-2}x^kq^{k}\sum_{r,s \in \mathbb{Z}} \text{sg}(r,s)(-1)^{r+s} q^{\frac{r^2}{2}+2trs+\frac{s^2}{2}+\frac{1+t+m+2-8t^2+2k}{2}r+\frac{1+t-m}{2}r + r(1-4t^2)} \\
&= -q^{\frac{-6t^2-t+2tm+m}{2}+1}\sum_{k=0}^{4t^2-2}x^kq^{k} f_{1,2t,1}(q^{2-4t^2 + \frac{t+m}{2} + k}, q^{1 + \frac{t-m}{2}}; q).
\end{aligned}
\end{equation}
Combining (\ref{expand}), (\ref{oneside}) and (\ref{otherside}) gives us (\ref{r1}). Finally, (\ref{r2}) follows from (\ref{theta1}), (\ref{t1}) and (\ref{r1}).
\end{proof}
We are now in a position to prove our first result. 
\begin{proof}[Proof of Theorem \ref{main1}]
Note that $\hat{g}_{t,m}(x) = \hat{g}_{t,m}(x;q)$ does not have poles and so we may write
\begin{equation} \label{L}
\hat{g}_{t,m}(x)  = \sum_{r \in \mathbb{Z}} (-1)^r q^{\frac{r^2}{8t^2}+\frac{t-m+2t^2-2tm}{8t^2}r}a_rx^{-r}
\end{equation}
for all $x \in \mathbb{C} \setminus \{ 0 \}$. Substituting (\ref{L}) into (\ref{r2}), we obtain 
\begin{equation} \label{ar}
 \begin{aligned}
\sum_{r \in \mathbb{Z}} (-1)^r q^{\frac{r^2}{8t^2} + \frac{t-m+2t^2-2tm}{8t^2}r-r}a_rx^{-r} &= x^{-4t^2}q^{\frac{t-m-2t^2-2tm}{2}}\sum_{r \in \mathbb{Z}}(-1)^rq^{\frac{r^2}{8t^2}+\frac{t-m+2t^2-2tm}{8t^2}r}a_rx^{-r} \\ 
&+ x^{-4t^2-\frac{t+m}{2}}q^{\frac{t-m-2t^2-2tm}{2}}(q)^3_{\infty}\sum_{i=1}^{2t}(-1)^iq^{\frac{i^2}{2}-\frac{1+t-m}{2}i}x^{2ti}\\
&+x^{-4t^2} q^{1-4t^2}\Theta(x;q) \\
& \qquad \qquad \qquad \quad \times \sum_{k=0}^{4t^2-2}x^{k}q^k f_{1,2t,1}(q^{2-4t^2 + \frac{t+m}{2} + k}, q^{1 + \frac{t-m}{2}}; q).
\end{aligned}
\end{equation}
Using 
\begin{equation} \label{id}
\binom{a-b}{2} = \binom{a}{2} - ab + \binom{b+1}{2}
\end{equation}
and (\ref{theta}), one can check that the last sum on the right-hand side of (\ref{ar}) can be written as
\begin{equation} \label{switch}
q^{1-4t^2} \sum_{r \in \mathbb{Z} }(-1)^r q^{\binom{r+1}{2}} \sum_{k=0}^{4t^2-2}(-1)^k q^{\binom{k+1-4t^2}{2} + r(k-4t^2)+k} f_{1,2t,1}(q^{2-4t^2 + \frac{t+m}{2} + k}, q^{1 + \frac{t-m}{2}}; q) x^{-r}.
\end{equation}
We now let $r \to r-4t^2$ in the first term on the right-hand side of (\ref{ar}), apply (\ref{switch}) and then compare coefficients of $x^{-r}$ in the resulting expressions to arrive at the recurrence relation
\begin{equation} \label{recur1}
a_r = a_{r-4t^2} + b_{r}^{\prime} + c_{r}^{\prime}
\end{equation}
where
\begin{equation*}
\begin{aligned}
b_{r}^{\prime} & := q^{1-4t^2 + \binom{r+1}{2}-\frac{r^2}{8t^2}-\frac{t-m+2t^2-2tm-8t^2}{8t^2}r-4t^2r}\sum_{k=0}^{4t^2-2}(-1)^k q^{\binom{k+1-4t^2}{2} + k(r +1)} \\
&  \qquad \qquad \qquad \qquad \qquad \qquad  \qquad \qquad \qquad \quad  \times f_{1,2t,1}(q^{2-4t^2 + \frac{t+m}{2} + k}, q^{1 + \frac{t-m}{2}}; q)
\end{aligned}
\end{equation*}
and
\begin{equation*}
c_{r}^{\prime} := (-1)^{i + \frac{t+m}{2}} (q)_{\infty}^3 q^{\frac{i^2}{2} - \frac{1+t-m}{2}i - \frac{\left(4t^2 + \frac{t+m}{2} - 2ti \right)^2}{8t^2} - \frac{t-m + 2t^2 - 2tm - 8t^2}{8t^2} \left(4t^2 + \frac{t+m}{2} - 2ti \right) + \frac{t-m-2t^2 - 2tm}{2}} 
\end{equation*}
if $r=4t^2 + \frac{t+m}{2} - 2ti$, $1\leq i \leq 2t$, and is 0 otherwise. Moreover, using (\ref{gen1}), (\ref{t1}) and Cauchy's integral formula applied to (\ref{L}), a short calculation gives
\begin{equation} \label{arform}
\begin{aligned}
a_r & = -\frac{1}{2 \pi i} q^{\frac{4t^2 - 1}{8t^2} r^2 - \frac{t-m+2t^2 - 2tm}{8t^2}r} \oint \sum_{\lambda \in \mathbb{Z}} (-1)^{\lambda} q^{\binom{\lambda + 1}{2} - \lambda r} \\
& \qquad \qquad \qquad \qquad  \qquad \qquad \qquad \quad \times \sum_{n,s \in \mathbb{Z}} \text{sg}(n,s) (-1)^{n+s} \frac{q^{\frac{n^2}{2} + 2tns + \frac{s^2}{2} + \frac{t+1+m}{2}n + \frac{t+1-m}{2}s}}{1-zq^n} \,\, dz
\end{aligned}
\end{equation}
where the integration is over a closed contour around $0$ in $\mathbb{C}$. Thus, as $| q |< 1$, 
\begin{equation} \label{keylim}
\lim_{r \to \pm \infty} a_r = 0.
\end{equation}
Now, observe that (\ref{recur1}) is equivalent to
\begin{equation} \label{recur2}
a_r - a_{r+4t^2} = b_r + c_r
\end{equation}
where $b_r := -b_{r+4t^2}^{\prime}$ and $c_{r} := -c_{r+4t^2}^{\prime}$. We now claim that
\begin{equation} \label{claim1}
a_r = \sum_{l \in \mathbb{Z}} \text{sg}(r,l) b_{r+4t^2l}.
\end{equation}
To deduce this, we let $\alpha_r := q^{\frac{r^2}{8t^2} + \frac{t-m + 2t^2 - 2tm}{8t^2}r} a_r$ and use (\ref{recur2}) to obtain
\begin{equation} \label{alpharecur}
\alpha_r = q^{-r-2t^2-\frac{t-m+2t^2-2tm}{8t^2}}\alpha_{r+4t^2}+q^{\frac{r^2}{8t^2}+\frac{t-m+2t^2-2tm}{8t^2}r}b_r+q^{\frac{r^2}{8t^2}+\frac{t-m+2t^2-2tm}{8t^2}r}c_r. 
\end{equation}
In fact, we will demonstrate 
\begin{equation*} \label{claim2}
\alpha_r = q^{\frac{r^2}{8t^2}+\frac{t-m+2t^2-2tm}{8t^2}r}\sum_{l\in\mathbb{Z}} \text{sg}(r,l) b_{r+4t^2l}
\end{equation*}
which clearly implies (\ref{claim1}). Let
\begin{equation*}
\Tilde{a}_r := \sum_{l \in \mathbb{Z}} \text{sg}(r,l)b_{r+4t^2l}
\end{equation*}
and
\begin{equation*}
\Tilde{\alpha}_r := q^{\frac{r^2}{8t^2}+\frac{t-m+2t^2-2tm}{8t^2}r}\Tilde{a}_r.
\end{equation*}
Then $\Tilde{a}_r$ and $\Tilde{\alpha}_r$ satisfy (\ref{recur2}) and (\ref{alpharecur}), respectively. The former follows from
\begin{equation} \label{ch1}
\begin{aligned}	
 \Tilde{a}_r - \Tilde{a}_{r+4t^2} &= \sum_{l\in\mathbb{Z}}\left( \text{sg}(r,l) - \text{sg}(r+4t^2,l-1) \right) b_{r+4t^2l}\\
&=\sum_{l\in\mathbb{Z}} \left( \delta(l) - \delta(r+1) - \cdots - \delta(r+4t^2) \right) b_{r+4t^2l}\\
&= b_r - \left(\delta(r+1) + \cdots + \delta(r+4t^2) \right) \sum_{n \equiv r \;(\text{mod}\;4t^2)} b_n
 \end{aligned}
 \end{equation}
 and
\begin{equation} \label{ch2}
\sum_{n\equiv r\;(\text{mod}\;4t^2)}b_n = \sum_{n\equiv r\;(\text{mod}\;4t^2)}(a_n -a_{n+4t^2}-c_n) = -\sum_{n\equiv r\;(\text{mod}\;4t^2)} c_n =-c_r
\end{equation}
where we have used (\ref{keylim}), the definitions of $c_r$ and $c_{r}^{\prime}$ and that $-t \leq m \leq 3t-2$. Now, since $\lim_{r \to \pm \infty} \alpha_r = 0$ and $\lim_{r \to \infty} \Tilde{\alpha}_r = 0$, we have
\begin{equation} \label{limdiff}
\lim_{r \to \infty} \left( \alpha_r - \Tilde{\alpha}_r \right) = 0.
\end{equation}
Finally, we compute
\begin{equation*}
\alpha_r - \Tilde{\alpha}_r - q^{r+2t^2+\frac{t-m+2t^2-2tm}{8t^2}} \left( \alpha_{r-4t^2} - \Tilde{\alpha}_{r-4t^2} \right)  = 0
\end{equation*}
which in combination with (\ref{limdiff}) implies that $\alpha_r = \Tilde{\alpha}_{r}$ and so $a_r = \Tilde{a}_r$. In total,
\begin{equation} \label{end1}
\begin{aligned}
\hat{g}_{t,m}(x) & = \sum_{r \in \mathbb{Z}} (-1)^ r q^{\frac{r^2}{8t^2} + \frac{t-m+2t^2-2tm}{8t^2}r} a_r x^{-r} \\
& = \sum_{r \in \mathbb{Z}} (-1)^r q^{\frac{r^2}{8t^2} + \frac{t-m+2t^2-2tm}{8t^2}r} \sum_{l \in \mathbb{Z}} \text{sg}(r,l) b_{r+4t^2 l} x^{-r} \\
& = - q^{1-4t^2} \sum_{r,l \in \mathbb{Z}} \text{sg}(r,l) (-1)^r q^{\binom{r+4t^2 (l+1) + 1}{2} - \frac{(r+4t^2 (l+1))^2}{8t^2} - \frac{t-m + 2t^2 - 2tm - 8t^2}{8t^2} (r + 4t^2 (l+1))} \\
& \quad \times q^{- 4t^2 (r + 4t^2 (l+1))} \sum_{k=0}^{4t^2 - 2} (-1)^k q^{\binom{k+1 - 4t^2}{2} + k(r+ 4t^2(l+1) + 1) + \frac{r^2}{8t^2} + \frac{t-m+2t^2-2tm}{8t^2}r} \\
&  \qquad \qquad \qquad \qquad \qquad \qquad \times f_{1,2t,1}(q^{2-4t^2 + \frac{t+m}{2} + k}, q^{1 + \frac{t-m}{2}}; q) x^{-r} \\
& = q^{1-t^2 - 8t^4 - \frac{t-m}{2} + tm} \sum_{k=0}^{4t^2 - 2} (-1)^{k+1} q^{\binom{k+1-4t^2}{2} + k + 4t^2 k} f_{1,2t,1}(q^{2-4t^2 + \frac{t+m}{2} + k}, q^{1 + \frac{t-m}{2}}; q) \\
& \quad \times \sum_{r, l \in \mathbb{Z}} \text{sg}(r,l) (-1)^r q^{\frac{r^2}{2} + (4t^2 - 1)rl + 2t^2 l^2 (4t^2 - 1) + (k+\frac{1}{2})r + (t^2 - \frac{t-m}{2} + tm + 4t^2 k) l} x^{-r} \\
& = q^{1-3t^2 - \frac{t-m}{2} + tm} \sum_{k=0}^{4t^2 - 2} (-1)^{k+1} q^{\binom{k+1}{2} + k} f_{1,2t,1}(q^{2-4t^2 + \frac{t+m}{2} + k}, q^{1 + \frac{t-m}{2}}; q) \\
& \qquad \qquad \qquad \qquad \qquad \qquad  \times f_{1,4t^2 - 1, 4t^2 (4t^2 - 1)}(x^{-1} q^{k+1}, -q^{4t^2 k - t^2 + tm - \frac{t-m}{2} + 8t^4}; q).
\end{aligned}
\end{equation}
Thus, (\ref{mf1}) follows from (\ref{realgen1}), (\ref{t1}) and (\ref{end1}). We now apply (\ref{mzftom}) and (\ref{theta2}) to deduce that $f_{1,2t,1}$ is a modular form and (\ref{ftofalse}) to obtain that $f_{1, 4t^2-1, 4t^2(4t^2-1)}$ is a false theta function. 
\end{proof}
For the second case, we begin with the following result.
\begin{proposition} For $t \in \mathbb{N}$ and $m \in \mathbb{Z}$, we have
\begin{equation} \label{r3}
\begin{aligned}
\mathcal{H}_{t}^{(m)}(qx;q) & = -x^{1-4t^2} q^{m-2t^2} \mathcal{H}_{t}^{(m)}(x;q) - x^{1-4t^2 + m-t} q^{m-2t^2} \frac{(q)_{\infty}^3}{\Theta(x;q)} \sum_{i=1}^{2t} (-1)^i q^{\binom{i}{2}} x^{2ti} \\
& - x^{1-4t^2} q^{1-4t^2} \sum_{i=0}^{4t^2-2} x^k q^k f_{1,2t,1}(q^{t-m+2 -4t^2 +k}, q; q)
\end{aligned}
\end{equation}
and
\begin{equation} \label{r4}
\begin{aligned}
\hat{\mathcal{H}}_{t}^{(m)}(qx;q) & = x^{-4t^2} q^{m-2t^2} \hat{\mathcal{H}}_{t}^{(m)}(x;q) + x^{-4t^2 + m-t} q^{m-2t^2} (q)_{\infty}^3 \sum_{i=1}^{2t} (-1)^i q^{\binom{i}{2}} x^{2ti} \\
& + x^{-4t^2} q^{1-4t^2} \Theta(x;q) \sum_{i=0}^{4t^2-2} x^k q^k f_{1,2t,1}(q^{t-m+2 -4t^2 +k}, q; q).
\end{aligned}
\end{equation}
\end{proposition}

\begin{proof} We first compute the sum
\begin{equation} \label{sum1}
x^{4t^2-1}q^{2t^2-mt}\sum_{r,s \in \mathbb{Z}} \text{sg}(r,s) (-1)^{r+s} q^{ \binom{r+1}{2} +2trs+ \binom{s+1}{2}+ (t-m)r} \frac{1-x^{1-4t^2}q^{(r+1)(1-4t^2)}}{1-xq^{r+1}}
\end{equation}
in two ways. Expanding the numerator yields
\begin{equation} \label{expand1}
x^{4t^2-1}q^{2t^2 - m} \mathcal{H}_t^{(m)}(qx;q) - \sum_{r,s \in \mathbb{Z}} \text{sg}(r,s)(-1)^{r+s}\frac{q^{\binom{r+1}{2} +2trs+ \binom{s+1}{2} + (t-m)r+ (r+1)(1-4t^2) + 2t^2 - m}}{1-xq^{r+1}}.
\end{equation}
Taking $(r,s) \to (r-1, s+2t)$ in the second sum in (\ref{expand1}) and using (\ref{prop1}), (\ref{theta2}) and (\ref{theta3}) yields
\begin{equation} \label{oneside1}
\begin{aligned}
 -\sum_{r,s \in \mathbb{Z}} \text{sg}(r-1,s+2t)(-1)^{r+s} & \frac{q^{\binom{r+1}{2} + 2trs+ \binom{s+1}{2} + (t-m)r}}{1-xq^r}  \\
 & \qquad \qquad \qquad = - \mathcal{H}_t^{(m)}(x;q) - \frac{x^{m-t} (q)_{\infty}^{3}}{\Theta(x;q)} \sum_{i=1}^{2t} (-1)^i q^{\binom{i}{2}} x^{2ti}. 
 \end{aligned}
\end{equation}
Alternatively, we use (\ref{geo}) to express (\ref{sum1}) as
\begin{equation} \label{otherside1}
\begin{aligned}
-q^{1-2t^2-m} \sum_{k=0}^{4t^2-2} x^k q^{k} f_{1,2t,1}(q^{t-m+2-4t^2+k}, q; q).
\end{aligned}
\end{equation}
Combining (\ref{expand1})--(\ref{otherside1}) gives us (\ref{r3}). Finally, (\ref{r4}) follows from (\ref{theta1}) and (\ref{r3}). 
\end{proof}
We can now prove our second main result.
\begin{proof}[Proof of Theorem \ref{main2}]
As $\hat{\mathcal{H}}_{t}^{(m)}(x) = \hat{\mathcal{H}}_{t}^{(m)}(x;q)$ does not have poles, we write
\begin{equation} \label{L1}
\hat{\mathcal{H}}_{t}^{(m)}(x)  = \sum_{r \in \mathbb{Z}} (-1)^r q^{\frac{r^2}{8t^2} + \frac{mr}{4t^2}} a_r x^{-r}
\end{equation}
for all $x \in \mathbb{C} \setminus \{ 0 \}$. Substituting (\ref{L1}) into (\ref{r4}), we obtain 
\begin{equation} \label{ar1}
\begin{aligned}
\sum_{r \in \mathbb{Z}} (-1)^r q^{\frac{r^2}{8t^2} + \frac{mr}{4t^2} - r} a_r x^{-r} &= x^{-4t^2} q^{m-2t^2} \sum_{r \in \mathbb{Z}}(-1)^r q^{\frac{r^2}{8t^2}+\frac{mr}{4t^2}} a_r x^{-r} \\ 
&+ x^{-4t^2 + m-t} q^{m-2t^2} (q)^3_{\infty} \sum_{i=1}^{2t}(-1)^i q^{\binom{i}{2}} x^{2ti}\\
&+ x^{-4t^2} q^{1-4t^2} \Theta(x;q) \sum_{k=0}^{4t^2-2} x^{k} q^k f_{1,2t,1}(q^{t-m+2-4t^2+k}, q; q).
\end{aligned}
\end{equation}
Using (\ref{theta}) and (\ref{id}), the last sum on the right-hand side of (\ref{ar1}) can be written as
\begin{equation} \label{switch1}
q^{1-4t^2} \sum_{r \in \mathbb{Z} }(-1)^r q^{\binom{r+1}{2}} \sum_{k=0}^{4t^2-2}(-1)^k q^{\binom{k+1-4t^2}{2} + r(k-4t^2)+k} f_{1,2t,1}(q^{t-m+2-4t^2+k}, q; q) x^{-r}.
\end{equation}
We now let $r \to r-4t^2$ in the first term on the right-hand side of (\ref{ar1}), apply (\ref{switch1}) and then compare coefficients of $x^{-r}$ in the resulting expressions to arrive at the recurrence relation
\begin{equation} \label{newrecur}
a_r = a_{r-4t^2} + b_{r}^{\prime} + c_{r}^{\prime}
\end{equation}
where
\begin{equation*}
b_{r}^{\prime} := q^{1-4t^2 + \binom{r+1}{2}-\frac{r^2}{8t^2}- \frac{mr}{4t} + r(1 - 4t^2)} \sum_{k=0}^{4t^2-2}(-1)^k q^{\binom{k+1-4t^2}{2} +k(r+1)} f_{1,2t,1}(q^{t-m+2-4t^2+k}, q; q) 
\end{equation*}
and
\begin{equation*}
c_{r}^{\prime} := (-1)^{i+m+t} (q)_{\infty}^3 q^{m-2t^2 + \binom{i}{2} - \frac{(4t^2 -m+t- 2ti)^2}{8t^2} - \frac{m}{4t^2}(4t^2 - m + t -2ti) + (4t^2 -m+t-2ti)}
\end{equation*}
if $r=4t^2-m+t-2ti$, $1\leq i \leq 2t$, and is 0 otherwise. Moreover, a similar computation as in (\ref{arform}) implies
\begin{equation} \label{newkeylim1}
\lim_{r \to \pm \infty} a_r = 0.
\end{equation}
Now, observe that (\ref{newrecur}) is equivalent to
\begin{equation} \label{newrecur2}
a_r - a_{r+4t^2} = b_r + c_r
\end{equation}
where $b_r := -b_{r+4t^2}^{\prime}$ and $c_{r} := -c_{r+4t^2}^{\prime}$. We now claim that
\begin{equation} \label{newclaim}
a_r = \sum_{l \in \mathbb{Z}} \text{sg}(r,l) b_{r+4t^2l}.
\end{equation}
To deduce this, we let $\alpha_r := q^{\frac{r^2}{8t^2} + \frac{mr}{4t^2}} a_r$ and use (\ref{newrecur2}) to obtain
\begin{equation} \label{newalpharecur}
\alpha_r = q^{-r-2t^2-m} \alpha_{r+4t^2} + q^{\frac{r^2}{8t^2}+\frac{mr}{4t^2}} b_r + q^{\frac{r^2}{8t^2}+\frac{mr}{4t^2}} c_r. 
\end{equation}
We will show 
\begin{equation*} \label{newclaim2}
\alpha_r = q^{\frac{r^2}{8t^2}+\frac{mr}{4t^2}} \sum_{l\in\mathbb{Z}} \text{sg}(r,l) b_{r+4t^2l}
\end{equation*}
which clearly implies (\ref{newclaim}). Let
\begin{equation*}
\Tilde{a}_r := \sum_{l \in \mathbb{Z}} \text{sg}(r,l)b_{r+4t^2l}
\end{equation*}
and 
\begin{equation*}
\Tilde{\alpha}_r := q^{\frac{r^2}{8t^2}+\frac{mr}{4t^2}} \Tilde{a}_r.
\end{equation*}
Then $\Tilde{a}_r$ and $\Tilde{\alpha}_r$ satisfy (\ref{newrecur2}) and (\ref{newalpharecur}), respectively, via the same calculation
as in (\ref{ch1}) and (\ref{ch2}) where we use (\ref{newkeylim1}) and $1-t \leq m \leq t$. In addition,
\begin{equation} \label{limdiff2}
\lim_{r \to \infty} \left( \alpha_r - \Tilde{\alpha}_r \right) = 0.
\end{equation}
Finally, we observe
\begin{equation*}
\alpha_r - \Tilde{\alpha}_r - q^{r+2t^2+m} \left( \alpha_{r-4t^2} - \Tilde{\alpha}_{r-4t^2} \right)  = 0
\end{equation*}
which in combination with (\ref{limdiff2}) implies that $\alpha_r = \Tilde{\alpha}_{r}$ and so $a_r = \Tilde{a}_r$.
In total,
\begin{equation} \label{end2}
\begin{aligned}
\hat{\mathcal{H}}_{t}^{(m)}(x) & = \sum_{r \in \mathbb{Z}} (-1)^r q^{\frac{r^2}{8t^2} + \frac{mr}{4t^2}} a_r x^{-r} \\
& = \sum_{r \in \mathbb{Z}} (-1)^r q^{\frac{r^2}{8t^2} + \frac{mr}{4t^2}} \sum_{l \in \mathbb{Z}} \text{sg}(r, l) b_{r+4t^2 l} x^{-r} \\
& = q^{1-m-8t^4} \sum_{r, l \in \mathbb{Z}} \text{sg}(r,l) (-1)^r q^{\binom{r}{2} + (4t^2 - 1)rl + 4t^2 (4t^2 - 1) \binom{l}{2} + r + (8t^4 - m)l} \\
& \qquad \qquad \qquad \times \sum_{k=0}^{4t^2 - 2} (-1)^k q^{\binom{k+1-4t^2}{2} + kr + 4t^2 kl + 4t^2 k + k} f_{1,2t,1}(q^{t-m+2-4t^2+k}, q; q) \\
& = q^{1-m-2t^2} \sum_{k=0}^{4t^2 - 2} (-1)^{k+1} q^{\binom{k+1}{2} + k} f_{1, 2t, 1}(q^{t-m+2-4t^2+k}, q; q) \\
&\qquad \qquad \qquad \times \sum_{r, l \in \mathbb{Z}} \text{sg}(r, l) (-1)^r q^{\binom{r}{2} + (4t^2 - 1)rl + 4t^2 (4t^2 - 1) \binom{l}{2} + (k+1)r + (8t^4 -m + 4t^2k)l} x^{-r} \\
& = q^{1-m-2t^2} \sum_{k=0}^{4t^2 - 2} (-1)^{k+1} q^{\binom{k+1}{2}+k} f_{1, 2t, 1}(q^{t-m+2-4t^2+k}, q; q) \\
& \qquad \qquad \qquad \qquad \times f_{1, 4t^2-1, 4t^2 (4t^2 - 1)}(x^{-1} q^{k+1}, - q^{8t^4 - m + 4t^2 k}; q).
\end{aligned}
\end{equation}
Thus, (\ref{mf2}) follows from (\ref{realgen2}), (\ref{piece}), (\ref{seconddecompose}), (\ref{t2}) and (\ref{end2}). We now apply (\ref{mzftom}) and (\ref{theta2}) to deduce that $f_{1,2t,1}$ is a modular form and (\ref{ftofalse}) to obtain that $f_{1, 4t^2-1, 4t^2(4t^2-1)}$ is a false theta function. 
\end{proof}
For the third case, we begin with the following result.
\begin{proposition} For $t \in \mathbb{N}$ and $m \in \mathbb{Z}$, we have
\begin{equation} \label{r7}
\begin{aligned}
{\kappa}_{t,m}(q^2 x;q) &= -x^{1-2t} q^{3-4t^2 - 3t + 4tm} {\kappa}_{t,m}(x;q) \\
& - x^{-2t^2 - t + 2} q^{4-4t^2 - 4t + 4tm} \frac{(q^2; q^2)_{\infty}^3}{\Theta(qx;q^2)} \sum_{i=1}^{2t} (-1)^i q^{\binom{i+1}{2} - 2mi + ti} x^{ti} \\
& - x^{1-2t^2} q^{3-6t^2} \sum_{k=0}^{2t^2 - 2} x^k q^{3k} f_{2,2t,1}(q^{2t+2 - 4t^2 + 2k}, q^{2m}; q)
\end{aligned}
\end{equation}
and
\begin{equation} \label{r8}
\begin{aligned}
\hat{\kappa}_{t,m}(q^2 x; q) & = x^{-2t^2} q^{2-4t^2 - 3t + 4tm} \hat{\kappa}_{t,m}(x;q) \\
& + x^{-2t^2 - t+1} q^{3 - 4t^2 - 4t + 4tm} (q^2; q^2)_{\infty}^3 \sum_{i=1}^{2t} (-1)^i q^{\binom{i+1}{2} - 2mi + ti} x^{ti} \\
& + x^{-2t^2} q^{2-6t^2} \Theta(qx; q^2)  \sum_{k=0}^{2t^2 - 2} x^k q^{3k} f_{2,2t,1}(q^{2t+2 - 4t^2 + 2k}, q^{2m}; q).
\end{aligned}
\end{equation}
\end{proposition}

\begin{proof} We first compute the sum
\begin{equation} \label{sum3}
x^{2t^2-1} q^{-3+4t^2+3t-4tm} \sum_{r,s \in \mathbb{Z}} \text{sg}(r,s) (-1)^{r+s} q^{2\binom{r}{2} + 2trs + \binom{s}{2} + 2tr + 2ms} \frac{1-x^{1-2t^2}q^{(2r+3)(1-2t^2)}}{1-xq^{2r+3}}
\end{equation}
in two ways. Expanding the numerator yields
\begin{equation} \label{expand3}
x^{2t^2-1}q^{-3+4t^2+3t-4tm} {\kappa}_{t,m}(q^2 x;q) - \sum_{r,s \in \mathbb{Z}} \text{sg}(r,s)(-1)^{r+s} \frac{q^{2\binom{r}{2} + 2trs + \binom{s}{2} + 2tr + 2ms - 2t^2 + 3t + 2r - 4t^2 r - 4tm}}{1-xq^{2r+3}}.
\end{equation}
Taking $(r,s) \to (r-1, s+2t)$ in the second sum in (\ref{expand3}) and using (\ref{prop1}), (\ref{theta2}) and (\ref{theta3}) yields
\begin{equation} \label{oneside3}
\begin{aligned}
&  -\sum_{r,s \in \mathbb{Z}} \text{sg}(r-1,s+2t)(-1)^{r+s}\frac{q^{2\binom{r}{2} + 2trs + \binom{s}{2} + 2tr + 2ms}}{1-xq^{2r+1}} \\
 & \qquad \qquad \qquad \qquad \qquad \qquad = - {\kappa}_{t,m}(x;q) - \frac{(qx)^{1-t} (q^2; q^2)_{\infty}^{3}}{\Theta(qx;q^2)} \sum_{i=1}^{2t} (-1)^i q^{\binom{i+1}{2} - 2mi + ti} x^{ti}. 
 \end{aligned}
\end{equation}
Alternatively, we use
\begin{equation*} \label{geo1}
\frac{1-x^{1-2t^2} q^{(2r+3)(1-2t^2)}}{1-xq^{2r+3}} = -x^{1-2t^2} q^{(2r+3)(1-2t^2)} \sum_{k=0}^{2t^2 - 2} x^k q^{(2r+3)k}
\end{equation*}
to express (\ref{sum3}) as
\begin{equation} \label{otherside3}
\begin{aligned}
-q^{-2t^2 + 3t - 4tm} \sum_{k=0}^{2t^2-2} x^k q^{3k} f_{2,2t,1}(q^{2t+2 - 4t^2 + 2k}, q^{2m}; q).
\end{aligned}
\end{equation}
Combining (\ref{expand3})--(\ref{otherside3}) gives us (\ref{r7}). Finally, (\ref{r8}) follows from (\ref{theta1}) and (\ref{r7}).
\end{proof}
We can now prove our third result.
\begin{proof}[Proof of Theorem \ref{main3}]
As $\hat{\kappa}_{t,m}(x) = \hat{\kappa}_{t,m}(x;q)$ does not have poles, we write
\begin{equation} \label{L3}
\hat{\kappa}_{t,m}(x) = \sum_{r \in \mathbb{Z}} (-1)^r q^{\frac{r^2}{2t^2} - \frac{2t^2 - 2 + 3t - 4tm}{2t^2} r} a_r x^{-r}
\end{equation}
for all $x \in \mathbb{C} \setminus \{ 0 \}$. Substituting (\ref{L3}) into (\ref{r8}), we obtain 
\begin{equation} \label{ar3}
\begin{aligned}
\sum_{r \in \mathbb{Z}} (-1)^r q^{\frac{r^2}{2t^2}  - \frac{2t^2 - 2 + 3t - 4tm}{2t^2} r - 2r} a_r x^{-r} &= x^{-2t^2} q^{2-4t^2 - 3t + 4tm} \sum_{r \in \mathbb{Z}}(-1)^r q^{\frac{r^2}{2t^2} - \frac{2t^2 - 2 + 3t - 4tm}{2t^2} r} a_r x^{-r} \\ 
&+ x^{-2t^2 - t + 1} q^{3 - 4t^2 - 4t + 4tm} (q^2; q^2)_{\infty}^3 \sum_{i=1}^{2t}(-1)^i q^{\binom{i+1}{2} - 2mi + ti} x^{ti} \\
&+ x^{-2t^2} q^{2-6t^2} \Theta(qx;q^2) \sum_{k=0}^{2t^2-2} x^{k} q^{3k} f_{2,2t,1}(q^{2t+2 - 4t^2 + 2k}, q^{2m}; q).
\end{aligned}
\end{equation}
Using (\ref{theta}) and (\ref{id}), the last sum on the right-hand side of (\ref{ar3}) can be written as
\begin{equation} \label{switch3}
q^{2-6t^2 + 4t^4} \sum_{r \in \mathbb{Z} }(-1)^r q^{r^2 - 4t^2 r} \sum_{k=0}^{2t^2-2}(-1)^k q^{k^2 + 3k + 2rk - 4t^2k}  f_{2,2t,1}(q^{2t+2 + (2r - 4t^2)k}, q^{2m}; q) x^{-r}.
\end{equation}
We now let $r \to r-2t^2$ in the first term on the right-hand side of (\ref{ar3}), apply (\ref{switch3}) and then compare coefficients of $x^{-r}$ in the resulting expressions to arrive at the recurrence relation
\begin{equation} \label{newrecur3}
a_r = a_{r-2t^2} + b_{r}^{\prime} + c_{r}^{\prime}
\end{equation}
where
\begin{equation*}
\begin{aligned}
b_{r}^{\prime} & := q^{2 + 4t^4 - 6t^2 + r^2 - 4t^2 r- \frac{r^2}{2t^2} + \frac{2t^2 - 2 + 3t - 4tm}{2t^2} r + 2r} \\
& \qquad \times \sum_{k=0}^{2t^2-2}(-1)^k q^{k^2 + 3k + (2r - 4t^2)k} f_{2,2t,1}(q^{2t+2 - 4t^2 + 2k}, q^{2m}; q)
\end{aligned}
\end{equation*}
and
\begin{equation*}
c_{r}^{\prime} := (-1)^{i+t+ti+1} (q^2;q^2)_{\infty}^3 q^{3 - 4t^2 - 4t + 4tm + \binom{i+1}{2} - 2mi + ti - \frac{(2t^2 + t - 1 - ti)^2}{2t^2} + \frac{6t^2 - 2 + 3t - 4tm}{2t^2}(2t^2 + t-1-ti)}
\end{equation*}
if $r=2t^2 + t-1-ti$, $1\leq i \leq 2t$, and is 0 otherwise. Moreover, a similar computation as in (\ref{arform}) implies
\begin{equation} \label{keylim1}
\lim_{r \to \pm \infty} a_r = 0.
\end{equation}
Now, observe that (\ref{newrecur3}) is equivalent to
\begin{equation} \label{newrecur4}
a_r - a_{r+2t^2} = b_r + c_r
\end{equation}
where $b_r := -b_{r+2t^2}^{\prime}$ and $c_{r} := -c_{r+2t^2}^{\prime}$. We now claim that
\begin{equation} \label{newclaim3}
a_r = \sum_{l \in \mathbb{Z}} \text{sg}(r,l) b_{r+2t^2l}.
\end{equation}
To deduce this, we let $\alpha_r := q^{\frac{r^2}{2t^2} - \frac{2t^2 - 2 + 3t - 4tm}{2t^2}r} a_r$ and use (\ref{newrecur4}) to obtain
\begin{equation} \label{newalpharecur2}
\alpha_r = q^{-2r-2+3t-4tm} \alpha_{r+2t^2} + q^{\frac{r^2}{2t^2} - \frac{2t^2 - 2 + 3t - 4tm}{2t^2}r} b_r + q^{\frac{r^2}{2t^2} - \frac{2t^2 - 2 + 3t - 4tm}{2t^2}r} c_r. 
\end{equation}
We will show 
\begin{equation*} \label{newclaim4}
\alpha_r = q^{\frac{r^2}{2t^2} - \frac{2t^2 - 2 + 3t - 4tm}{2t^2}r} \sum_{l\in\mathbb{Z}} \text{sg}(r,l) b_{r+2t^2l}
\end{equation*}
which clearly implies (\ref{newclaim3}). Let
\begin{equation*}
\Tilde{a}_r := \sum_{l \in \mathbb{Z}} \text{sg}(r,l)b_{r+2t^2l}
\end{equation*}
and
\begin{equation*}
\Tilde{\alpha}_r := q^{\frac{r^2}{2t^2} - \frac{2t^2 - 2 + 3t - 4tm}{2t^2}r} \Tilde{a}_r.
\end{equation*}
Then $\Tilde{a}_r$ and $\Tilde{\alpha}_r$ satisfy (\ref{newrecur3}) and (\ref{newalpharecur2}), respectively, via the same calculation
as in (\ref{ch1}) and (\ref{ch2}) with $r+ 4t^2$ and $r+4t^2 l$ replaced with $r + 2t^2$ and $r+2t^2 l$, respectively, and (\ref{keylim1}). So,
\begin{equation} \label{limdiff3}
\lim_{r \to \infty} \left( \alpha_r - \Tilde{\alpha}_r \right) = 0.
\end{equation}
Finally, we observe
\begin{equation*}
\alpha_r - \Tilde{\alpha}_r - q^{2r+2-3t+4tm} \left( \alpha_{r-2t^2} - \Tilde{\alpha}_{r-2t^2} \right)  = 0
\end{equation*}
which in combination with (\ref{limdiff3}) implies that $\alpha_r = \Tilde{\alpha}_{r}$ and so $a_r = \Tilde{a}_r$. In total,
\begin{equation} \label{end4}
\begin{aligned}
\hat{\kappa}_{t,m}(x) & = \sum_{r \in \mathbb{Z}} (-1)^ r  q^{\frac{r^2}{2t^2} - \frac{2t^2 - 2 + 3t - 4tm}{2t^2}r} a_r x^{-r} \\
& = \sum_{r \in \mathbb{Z}} (-1)^r  q^{\frac{r^2}{2t^2} - \frac{2t^2 - 2 + 3t - 4tm}{2t^2}r} \sum_{l \in \mathbb{Z}} \text{sg}(r,l) b_{r+2t^2 l} x^{-r} \\
& = q^{2+4t^4-6t^2} \sum_{r,l \in \mathbb{Z}} \text{sg}(r,l) (-1)^r q^{(r+2t^2 (l+1))^2 -4t^2 (r + 2t^2(l+1)) - \frac{(r+2t^2 (l+1))^2}{2t^2}} \\
& \quad \times  q^{\frac{2t^2 - 2 + 3t - 4tm}{2t^2} (r+2t^2 (l+1)) + 2(r+2t^2(l+1))} \sum_{k=0}^{2t^2 - 2} (-1)^{k+1} q^{k^2 + 3k + (2(r+2t^2(l+1)) - 4t^2)k + \frac{r^2}{2t^2}} \\
& \qquad \qquad  \qquad \qquad \qquad \qquad \qquad \qquad \qquad \quad \times q^{- \frac{2t^2 - 2 + 3t - 4tm}{2t^2}r} f_{2,2t,1}(q^{2t+2 - 4t^2 + 2k}, q^{2m}; q) \\
& = q^{-2t^2 + 3t - 4tm} \sum_{k=0}^{2t^2 - 2} (-1)^{k+1} q^{k^2 + 3k} f_{2,2t,1}(q^{2t+2 - 4t^2 + 2k}, q^{2m}; q) \\
& \qquad \qquad \qquad \qquad \times \sum_{r, l \in \mathbb{Z}} \text{sg}(r,l) (-1)^r q^{r^2 + (4t^2 - 2)rl + 2t^2 l^2 (2t^2 - 1) + 2kr + (4t^2 k + 2t^2 - 2 + 3t - 4tm) l} x^{-r} \\
& = q^{-2t^2 + 3t - 4tm} \sum_{k=0}^{2t^2 - 2} (-1)^{k+1} q^{k^2 + 3k} f_{2,2t,1}(q^{2t+2 - 4t^2 + 2k}, q^{2m}; q) \\
& \qquad \qquad \qquad \qquad \times f_{1, 2t^2 - 1, 2t^2 (2t^2 - 1)}(x^{-1} q^{2k +1}, -q^{4t^2 k - 2 + 3t - 4tm + 4t^4}; q^2).
\end{aligned}
\end{equation}
Thus, (\ref{mf3}) follows from (\ref{realgen4}), (\ref{kappa}), (\ref{split}), (\ref{t3}) and (\ref{end4}). We now apply (\ref{mzftom}) and (\ref{theta2}) to deduce that $f_{2, 2t, 1}$ is a modular form and (\ref{ftofalse}) to obtain that $f_{1,2t^2 -1, 2t^2(2t^2-1)}$ is a false theta function. 
\end{proof}
For our last case, we start with the following result.
\begin{proposition} For $t \in \mathbb{N}$ and $m \in \mathbb{Z}$, we have
\begin{equation} \label{r9}
\begin{aligned}
p_{t,m}(qx;q) &= -x^{1-4t^2}q^{(m-1)(2t-1)} p_{t,m}(x;q) \\
& - x^{1-4t^2 - t - m + 1} q^{(m-1)(2t-1)} \frac{(q)_{\infty}^{3}}{\Theta(x;q)} \sum_{i=1}^{2t}(-1)^i q^{\binom{i+1}{2} - i(t+m)} x^{2ti}\\
&- x^{1-4t^2} q^{1-4t^2} \sum_{k=0}^{4t^2-2}x^k q^{k} f_{1,2t,1}(q^{t+m+1-4t^2 +k}, q^{t+m}; q)
\end{aligned}
\end{equation}
and
\begin{equation} \label{r10}
\begin{aligned}
\hat{p}_{t,m}(qx;q) &= x^{-4t^2} q^{(m-1)(2t-1)} \hat{p}_{t,m}(x;q) \\
& + x^{-4t^2 - t - m +1} q^{(m-1)(2t-1)} (q)^3_{\infty} \sum_{i=1}^{2t}(-1)^i q^{\binom{i+1}{2} - i(t+m)} x^{2ti}\\
& + x^{-4t^2} q^{1-4t^2}\Theta(x;q) \sum_{k=0}^{4t^2-2} x^{k} q^{k} f_{1,2t,1}(q^{t + m+1-4t^2 +k}, q^{t+m}; q).
\end{aligned}
\end{equation}
\end{proposition}

\begin{proof} We first compute the sum
\begin{equation} \label{sum4}
x^{4t^2-1} q^{-(m-1)(2t-1)} \sum_{r,s \in \mathbb{Z}} \text{sg}(r,s) (-1)^{r+s} q^{\binom{r}{2}+2trs+\binom{s}{2} + (t+m)(r+s)} \frac{1-x^{1-4t^2}q^{(r+1)(1-4t^2)}}{1-xq^{r+1}}
\end{equation}
in two ways. Expanding the numerator yields
\begin{equation} \label{expand4}
\begin{aligned}
x^{4t^2-1} q^{-(m-1)(2t-1)} p_{t,m}(qx;q) - & \sum_{r,s \in \mathbb{Z}}\text{sg}(r,s)(-1)^{r+s} \\
& \qquad \quad \times \frac{q^{\binom{r}{2} + 2trs + \binom{s}{2} + (t+m+1-4t^2)r + (t+m)s - (m-1)(2t-1) + 1 -4t^2}}{1-xq^{r+1}}.
\end{aligned}
\end{equation}
Taking $(r,s) \to (r-1, s+2t)$ in the second sum in (\ref{expand4}) and using (\ref{prop1}), (\ref{theta2}) and (\ref{theta3}) leads to
\begin{equation} \label{oneside4}
\begin{aligned}
& -\sum_{r,s \in \mathbb{Z}} \text{sg}(r-1,s+2t) (-1)^{r+s} \frac{q^{\binom{r}{2}+2trs+\binom{s}{2} + (t+m)r + (t+m)s}}{1-xq^r} \\
& = - \sum_{r,s \in \mathbb{Z}} \text{sg}(r,s)(-1)^{r+s}\frac{q^{\binom{r}{2}+2trs+\binom{s}{2}+(t+m)r + (t+m)s}}{1-xq^r} + \frac{1}{1-x}\sum_{s\in\mathbb{Z}}(-1)^s q^{\binom{s}{2} + (t+m)s} \\
& - \sum_{i=1}^{2t}\sum_{r \in \mathbb{Z}}(-1)^{r-i}\frac{q^{\binom{r}{2}-2tri+\binom{i+1}{2}+(t+m)r - (t+m)i}}{1-xq^r} \\
& = - p_{t,m}(x;q) - \frac{x^{-t-m+1} (q)_{\infty}^{3}}{\Theta(x;q)}\sum_{i=1}^{2t}(-1)^i q^{\binom{i+1}{2} - (t+m)i} x^{2ti}. 
\end{aligned}
\end{equation}
Alternatively, we use (\ref{geo}) to express (\ref{sum4}) as
\begin{equation} \label{otherside4}
\begin{aligned}
&  -q^{-(m-1)(2t-1) + 1-4t^2} \sum_{k=0}^{4t^2-2} x^k q^{k} \sum_{r,s \in \mathbb{Z}} \text{sg}(r,s)(-1)^{r+s} q^{\binom{r}{2} + 2trs + \binom{s}{2} + (t+m)(r+s) + r(1-4t^2 +k)} \\
&= -q^{-(m-1)(2t-1) + 1 - 4t^2} \sum_{k=0}^{4t^2-2} x^k q^{k} f_{1,2t,1}(q^{t+m+1-4t^2 + k}, q^{t+m}; q).
\end{aligned}
\end{equation}
Combining (\ref{expand4})--(\ref{otherside4}) gives us (\ref{r9}). Finally, (\ref{r10}) follows from (\ref{theta1}), (\ref{t4}) and (\ref{r9}).
\end{proof}
We can now prove our final result.
\begin{proof}[Proof of Theorem \ref{main4}]
As $\hat{p}_{t,m}(x) = \hat{p}_{t,m}(x;q)$ does not have poles, we write
\begin{equation} \label{laur}
\hat{p}_{t,m}(x)  = \sum_{r \in \mathbb{Z}} (-1)^r q^{\frac{r^2}{8t^2}+\frac{2t^2 + (m-1)(2t-1)}{4t^2}r} a_r x^{-r}
\end{equation}
for all $x \in \mathbb{C} \setminus \{0\}$. Substituting (\ref{laur}) into (\ref{r10}), we obtain
\begin{equation} \label{recur}
\begin{aligned}
\sum_{r \in \mathbb{Z}} (-1)^r q^{\frac{r^2}{8t^2} + \frac{2t^2 + (m-1)(2t-1)}{4t^2}r-r} a_r x^{-r} &= x^{-4t^2} q^{(m-1)(2t-1)} \sum_{r \in \mathbb{Z}} (-1)^r q^{\frac{r^2}{8t^2} + \frac{2t^2 + (m-1)(2t-1)}{4t^2}r} a_r x^{-r} \\ 
&+ x^{-4t^2 + 1 - t - m} q^{(m-1)(2t-1)}(q)^3_{\infty}\sum_{i=1}^{2t}(-1)^i q^{\binom{i+1}{2} - (t+m)i} x^{2ti}\\
&+x^{-4t^2} q^{1-4t^2}\Theta(x;q) \\
& \qquad \qquad \qquad \quad \times \sum_{k=0}^{4t^2-2}x^{k}q^k f_{1,2t,1}(q^{t+m+1-4t^2 +k}, q^{t+m}; q).
\end{aligned}
\end{equation}
Using (\ref{theta}) and (\ref{id}), the last sum on the right-hand side of (\ref{recur}) can be written as
\begin{equation} \label{rewrite}
q^{1-4t^2} \sum_{r \in \mathbb{Z} }(-1)^r q^{\binom{r+1}{2}} \sum_{k=0}^{4t^2-2}(-1)^k q^{\binom{k+1-4t^2}{2} + r(k-4t^2)+k} f_{1,2t,1}(q^{t+m+1-4t^2 + k}, q^{t+m}; q) x^{-r}.
\end{equation}
We now let $r \to r-4t^2$ in the first term on the right-hand side of (\ref{recur}), apply (\ref{rewrite}) and then compare coefficients of $x^{-r}$ in the resulting expressions to arrive at the recurrence relation
\begin{equation*} \label{newrecur1}
a_r = a_{r-4t^2} + b_r^{\prime} + c_r^{\prime}
\end{equation*}
where
\begin{equation*} \label{defbprime}
\begin{aligned}
b_{r}^{\prime} & := q^{1-4t^2 + \binom{r+1}{2}-\frac{r^2}{8t^2}-\frac{2t^2 + (m-1)(2t-1) - 4t^2}{4t^2}r-4t^2r} \sum_{k=0}^{4t^2-2}(-1)^k q^{\binom{k+1-4t^2}{2} + k(r +1)} \\
&  \qquad \qquad \qquad \qquad \qquad \qquad  \qquad \qquad \qquad \quad  \times f_{1,2t,1}(q^{t+m+1-4t^2+k}, q^{t+m}; q)
\end{aligned}
\end{equation*}
and
\begin{equation*} \label{defcprime}
c_{r}^{\prime} := (-1)^{i+1+t+m} (q)_{\infty}^3 q^{(m-1)(2t-1) + \binom{i+1}{2} - (t+m)i - \frac{(4t^2 -1 + t + m - 2ti)^2}{8t^2} - \frac{2t^2 + (m-1)(2t-1) - 4t^2}{4t^2} (4t^2 - 1 + t + m - 2ti)} 
\end{equation*}
if $r=t+m-1-2ti$, $1 \leq i \leq 2t$ and is $0$ otherwise. As before, we have
\begin{equation*}
a_r = \sum_{l \in \mathbb{Z}} \text{sg}(r,l) b_{r +4t^2 l}
\end{equation*}
where $b_r := - b_{r+4t^2}^{\prime}$ and so in total
\begin{equation} \label{lastone}
\begin{aligned}
\hat{p}_{t,m}(x) & = \sum_{r \in \mathbb{Z}} (-1)^ r q^{\frac{r^2}{8t^2} + \frac{2t^2 + (m-1)(2t-1)}{4t^2}r} a_r x^{-r} \\
& = \sum_{r \in \mathbb{Z}} (-1)^r q^{\frac{r^2}{8t^2} + \frac{2t^2 + (m-1)(2t-1)}{4t^2}r} \sum_{l \in \mathbb{Z}} \text{sg}(r,l) b_{r+4t^2 l} x^{-r} \\
& = q^{1-4t^2 - (m-1)(2t-1)} \sum_{k=0}^{4t^2 - 2} (-1)^{k+1} q^{\binom{k+1}{2} + k} f_{1,2t,1}(q^{t+m+1-4t^2 +k}, q^{t+m}; q) \\
& \quad \times \sum_{r, l \in \mathbb{Z}} \text{sg}(r,l) (-1)^r q^{\frac{r^2}{2} + (4t^2 - 1)rl + 2t^2 l^2 (4t^2 - 1) + (k+\frac{1}{2})r + (-(m-1)(2t-1) + 4t^2 k) l} x^{-r} \\
& = q^{1-4t^2 - (m-1)(2t-1)} \sum_{k=0}^{4t^2 - 2} (-1)^{k+1} q^{\binom{k+1}{2} + k} f_{1,2t,1}(q^{t+m+1-4t^2}, q^{t+m}; q) \\
& \qquad \qquad \qquad \qquad \qquad \qquad  \times f_{1,4t^2 - 1, 4t^2 (4t^2 - 1)}(x^{-1} q^{k+1}, -q^{8t^4 - 2t^2 - (m-1)(2t-1) + 4t^2k}; q).
\end{aligned}
\end{equation}
Thus, (\ref{mf4}) follows from (\ref{realgen5}), (\ref{t4}) and (\ref{lastone}). We now apply (\ref{mzftom}) and (\ref{theta2}) to deduce that $f_{1,2t,1}$ is a modular form and (\ref{ftofalse}) to obtain that $f_{1, 4t^2-1, 4t^2(4t^2-1)}$ is a false theta function. 
\end{proof}

\section{Applications}

\subsection{Recovering (\ref{pt1})} We set $t=m=1$ in Theorem \ref{main1} to obtain
\begin{equation} \label{s1}
\begin{aligned}
\mathcal{U}_1^{(1)}(x;q) & = \frac{(1-x)}{\Theta(x;q)} \frac{q^{-1}}{(q)_{\infty}^2} \Biggl( -f_{1,2,1}(q^{-1}, q; q) f_{1,3,12}(x^{-1}q, -q^{8}; q) \\
& + q^2 f_{1,2,1}(1,q;q) f_{1,3,12}(x^{-1}q^2, -q^{12}; q) - q^5 f_{1,2,1}(q,q;q) f_{1,3,12}(x^{-1} q^3, -q^{16}; q) \Biggr).
\end{aligned}
\end{equation}
Using \cite[Eq. (1.7)]{hm} and (\ref{theta1}), one can show that $f_{1,2,1}(q^{-1}, q; q) = -q (q)_{\infty}^2$, $f_{1,2,1}(1,q;q) = 0$ and $f_{1,2,1}(q,q;q) = (q)_{\infty}^2$ and so (\ref{s1}) becomes
\begin{equation} \label{s2}
\mathcal{U}_1^{(1)}(x;q) = \frac{(1-x)}{\Theta(x;q)} \Bigl( f_{1,3,12}(x^{-1} q, -q^8; q) - q^4 f_{1,3,12}(x^{-1} q^3, -q^{16}; q) \Bigr).
\end{equation}
By (\ref{sgo}), (\ref{theta}), (\ref{ftofalse}), (\ref{theta1}), the quintuple product identity
\begin{equation} \label{qpi}
\sum_{k \in \mathbb{Z}} q^{\frac{k(3k-1)}{2}} x^{3k} (1 - xq^k) = (q,x,q/x)_{\infty} (qx^2, qx^{-2}; q^2)_{\infty},
\end{equation}
(\ref{s2}) and some simplifications, we have
\begin{equation} \label{s3}
\begin{aligned}
\mathcal{U}_1^{(1)}(x;q) & = (1-x) \Biggl( \frac{1}{2} \sum_{r \geq 0} (-1)^r x^{3r} q^{\frac{r(3r+1)}{2}} (1 - x^2 q^{2r+1})  - \frac{1}{2} \sum_{r < 0} (-1)^r x^{3r} q^{\frac{r(3r+1)}{2}} (1 - x^2 q^{2r+1}) \\
& - \frac{1}{2\Theta(x;q)} \sum_{t=0}^{11} (-1)^t x^{-t} q^{\binom{t}{2}} \frac{\Theta(q^t; q^4) \Theta(q^{4+2t}; q^8)}{(q^8; q^8)_{\infty}} \sum_{r \in \mathbb{Z}} \text{sg}(r) (q^{-18 + 3t} x^{-12})^r q^{36 \binom{r+1}{2}} \Biggr).
\end{aligned}
\end{equation}
For $t$ even, one of the theta functions in the numerator of the third term on the right-hand side of (\ref{s3}) is zero. Thus, it suffices to consider $t$ odd. There are two steps. We first replace $t$ with $4t+1$ and let $0 \leq t \leq 2$ in the third term on the right-hand side of (\ref{s3}). This eventually yields
\begin{equation} \label{s4}
\frac{x^{-1}}{2(x)_{\infty} (q/x)_{\infty}} \sum_{k \in \mathbb{Z}} \text{sg}(k) x^{-4k} q^{\binom{2k+1}{2}}.
\end{equation}
Next, we replace $t$ with $4t+3$ and let $0 \leq t \leq 2$ in the third term on the right-hand side of (\ref{s3}). This gives
\begin{equation} \label{s5}
-\frac{x^{-3}}{2(x)_{\infty} (q/x)_{\infty}} \sum_{k \in \mathbb{Z}} \text{sg}(k) x^{-4k} q^{\binom{2k+2}{2}}.
\end{equation}
Combining (\ref{s4}) and (\ref{s5}), performing the shift $k \to -k-1$ and using $\text{sg}(-k-1) = -\text{sg}(k)$ leads to 
\begin{equation} \label{s6}
\frac{x}{2 (x)_{\infty} (q/x)_{\infty}} \sum_{k \in \mathbb{Z}} (-1)^k \text{sg}(k) x^{2k} q^{\binom{k+1}{2}}.
\end{equation}
Now, after inserting (\ref{s6}) into (\ref{s3}), using (\ref{sgo}) and rearranging terms, we have
\begin{equation} \label{s7}
\begin{aligned}
\mathcal{U}_1^{(1)}(x;q) & = (1-x) \Biggl( \sum_{r \geq 0} (-1)^r x^{3r} q^{\frac{r(3r+1)}{2}} (1 - x^2 q^{2r+1}) - \frac{1}{2} \sum_{r \in \mathbb{Z}} (-1)^r x^{3r} q^{\frac{r(3r+1)}{2}} (1 - x^2 q^{2r+1})\\
& + \frac{x}{(x)_{\infty} (q/x)_{\infty}} \sum_{k \geq 0} (-1)^k x^{2k} q^{\binom{k+1}{2}} - \frac{x}{2(x)_{\infty} (q/x)_{\infty}} \sum_{k \in \mathbb{Z}} (-1)^k x^{2k} q^{\binom{k+1}{2}} \Biggr).
\end{aligned}
\end{equation}
Applying (\ref{theta}) twice, then (\ref{theta1}) and (\ref{qpi}) to the second sum on the right-hand side of (\ref{s7}) implies
\begin{equation} \label{s8}
\sum_{r \in \mathbb{Z}} (-1)^r x^{3r} q^{\frac{r(3r+1)}{2}} (1 - x^2 q^{2r+1}) = x^{-1} \frac{(q)_{\infty} \Theta(x^2;q)}{\Theta(x;q)}
\end{equation}
while (\ref{theta}) and (\ref{theta1}) yield
\begin{equation} \label{s9}
\sum_{k \in \mathbb{Z}} (-1)^k x^{2k+1} q^{\binom{k+1}{2}} = -x^{-1} \Theta(x^2;q).
\end{equation}
Thus, (\ref{pt1}) follows from (\ref{s7})--(\ref{s9}) and then replacing $x$ with $-x$.

\subsection{Recovering (\ref{pt2})} We first set $t=m=1$ in Theorem \ref{main2}, then use \cite[Eq. (1.7)]{hm} and (\ref{theta1}) to deduce that $f_{1,2,1}(q^{-2}, q, q)=-q^2 (q)_{\infty}^2$, $f_{1,2,1}(q^{-1}, q; q) = -q (q)_{\infty}^2$ and $f_{1,2,1}(1, q; q)=0$ in order to obtain 
\begin{equation} \label{ss1}
\begin{aligned}
W_{1}^{(1)}(x;q) & = -\frac{(1-x)}{(x)_{\infty} (q/x)_{\infty}} + \frac{(1-x)}{\Theta(x;q)} \Biggl( f_{1,3,12}(x^{-1}q, -q^7; q) + f_{1,3,12}(xq, -q^7; q) \\
& \qquad \qquad \qquad \qquad \qquad \qquad \qquad - q f_{1,3,12}(x^{-1} q^2, -q^{11}; q) - q f_{1,3,12}(xq^2, -q^{11}; q) \Biggr).
\end{aligned}
\end{equation}
By (\ref{ftofalse}), (\ref{qpi}) and some simplifications, we have
\begin{equation} \label{ss2}
\begin{aligned}
& f_{1,3,12}(xq, -q^7; q) - q f_{1,3,12}(xq^2, -q^{11}; q) \\
& \qquad \qquad = \frac{\Theta(x;q)}{2} \Biggl( -x^{-1} \sum_{r \in \mathbb{Z}} \text{sg}(r) (-1)^r x^{-3r} q^{3\binom{r+1}{2} - 2r} - x^{-2} \sum_{r \in \mathbb{Z}} \text{sg}(r) (-1)^r x^{-3r} q^{3 \binom{r+1}{2} - r} \Biggr) \\
& \qquad \qquad + \frac{1}{2} \sum_{t=0}^{11} (-1)^t x^t q^{\binom{t}{2}} \frac{\Theta(q^{t+1}; q^4) \Theta(q^{6+2t}; q^8)}{(q^8; q^8)_{\infty}} \sum_{r \in \mathbb{Z}} \text{sg}(r) (q^{-15 + 3t} x^{12})^r q^{36 \binom{r+1}{2}}. 
\end{aligned} 
\end{equation}
We now let $r \to -r-1$ in each of the first two terms on the right-hand side of (\ref{ss2}), use $\text{sg}(-r-1) = - \text{sg}(r)$ and simplify to obtain
\begin{equation} \label{ss3}
-\frac{\Theta(x;q)}{2} \sum_{r \in \mathbb{Z}} \text{sg}(r) (-1)^r x^{3r+1} q^{\frac{(r+1)(3r+2)}{2}} (1 + xq^{r+1}).
\end{equation}
For $t$ odd, one of the theta functions in the numerator of the third term on the right-hand side of (\ref{ss2}) is zero. Thus, it suffices to consider $t$ even. There are two steps. We first replace $t$ with $4t$ and let $0 \leq t \leq 2$ in the third term on the right-hand side of (\ref{ss2}). This eventually yields
\begin{equation} \label{ss4}
\frac{(q)_{\infty}}{2} \sum_{k \in \mathbb{Z}} \text{sg}(k) x^{4k} q^{\binom{2k+1}{2}}.
\end{equation}
Next, we replace $t$ with $4t+2$ and let $0 \leq t \leq 2$ in the third term on the right-hand side of (\ref{ss2}). This gives
\begin{equation} \label{ss5}
-\frac{(q)_{\infty}}{2} \sum_{k \in \mathbb{Z}} \text{sg}(k) x^{4k+2} q^{\binom{2k+2}{2}}.
\end{equation}
We now insert the sum of (\ref{ss4}) and (\ref{ss5}) along with (\ref{ss3}) into (\ref{ss2}) to obtain
\begin{equation} \label{ss6}
\begin{aligned}
& f_{1,3,12}(xq, -q^7; q) - q f_{1,3,12}(xq^2, -q^{11}; q) \\
& \qquad \quad = -\frac{\Theta(x;q)}{2} \sum_{r \in \mathbb{Z}} \text{sg}(r) (-1)^r x^{3r+1} q^{\frac{(r+1)(3r+2)}{2}} (1 + xq^{r+1}) + \frac{(q)_{\infty}}{2} \sum_{k \in \mathbb{Z}} \text{sg}(k) (-1)^k x^{2k} q^{\binom{k+1}{2}}.
\end{aligned}
\end{equation}
A similar computation yields
\begin{equation} \label{ss7}
\begin{aligned}
& f_{1,3,12}(x^{-1} q, -q^7; q) - q f_{1,3,12}(x^{-1} q^2, -q^{11}; q) \\
& \qquad \quad = \frac{\Theta(x;q)}{2} \sum_{r \in \mathbb{Z}} \text{sg}(r) (-1)^r x^{3r} q^{\frac{r(3r-1)}{2}} (1 + xq^{r}) + \frac{(q)_{\infty}}{2} \sum_{k \in \mathbb{Z}} \text{sg}(k) (-1)^k x^{-2k} q^{\binom{k+1}{2}}.
\end{aligned}
\end{equation}
After combining (\ref{ss6}) and (\ref{ss7}), (\ref{ss1}), using (\ref{sgo}) and rearranging terms, we have
\begin{equation} \label{ss8}
\begin{aligned}
W_{1}^{(1)}(x;q) & =  -\frac{(1-x)}{(x)_{\infty} (q/x)_{\infty}} + (1-x) \Biggl( -\sum_{r \geq 0} (-1)^r x^{3r+1} q^{\frac{(r+1)(3r+2)}{2}} ( 1+ xq^{r+1}) \\
& + \frac{1}{2} \sum_{r \in \mathbb{Z}} (-1)^r x^{3r+1} q^{\frac{(r+1)(3r+2)}{2}} ( 1+ xq^{r+1}) + \sum_{r \geq 0} (-1)^r x^{3r} q^{\frac{r(3r-1)}{2}} (1 +xq^r) \\
& - \frac{1}{2} \sum_{r \in \mathbb{Z}} (-1)^r x^{3r} q^{\frac{r(3r-1)}{2}} (1 + xq^r) + \frac{1+x^2}{(x)_{\infty} (q/x)_{\infty}} \sum_{k \geq 0} (-1)^k x^{2k} q^{\binom{k+1}{2}} \\
& - \frac{1+x^2}{2(x)_{\infty} (q/x)_{\infty}} \sum_{k \in \mathbb{Z}} (-1)^k x^{2k} q^{\binom{k+1}{2}} \Biggr).
\end{aligned}
\end{equation}
Applying (\ref{theta}) twice, then (\ref{qpi}) to each of the second and fourth sums and (\ref{theta}) to the sixth sum on the right-hand side of (\ref{ss8}), then using (\ref{theta1}) and cancelling the resulting theta functions leads to
\begin{equation} \label{ss9}
\begin{aligned}
W_{1}^{(1)}(x;q) & =  -\frac{(1-x)}{(x)_{\infty} (q/x)_{\infty}} + (1-x) \Biggl( -\sum_{r \geq 0} (-1)^r x^{3r+1} q^{\frac{(r+1)(3r+2)}{2}} ( 1+ xq^{r+1}) \\
& + \sum_{r \geq 0} (-1)^r x^{3r} q^{\frac{r(3r-1)}{2}} (1 +xq^r) + \frac{1+x^2}{(x)_{\infty} (q/x)_{\infty}} \sum_{k \geq 0} (-1)^k x^{2k} q^{\binom{k+1}{2}} \Biggr).
\end{aligned}
\end{equation}
Finally, we shift $r \to r-1$ in the first sum, remove the $r=0$ term from the second sum and the $k=0$ term from the third sum on the right-hand side of (\ref{ss9}). In total, we have
\begin{equation*}
\begin{aligned}
W_{1}^{(1)}(x;q) & =  -\frac{(1-x)}{(x)_{\infty} (q/x)_{\infty}} + (1-x) \Biggl( 1 + x + (1+x^2) \sum_{r \geq 1} (-1)^r x^{3r-2} q^{\frac{r(3r-1)}{2}} (1 + xq^r)  \\
& + \frac{1+x^2}{(x)_{\infty} (q/x)_{\infty}} + \frac{1+x^2}{(x)_{\infty} (q/x)_{\infty}} \sum_{k \geq 1} (-1)^k x^{2k} q^{\binom{k+1}{2}} \Biggr)
\end{aligned}
\end{equation*}
which is (\ref{pt2}) after simplification.

\subsection{Recovering (\ref{pt3})} We first set $t=m=1$ in Theorem \ref{main3}, then use \cite[Corollary 3.11]{mps} to deduce that $f_{2,2,1}(1,q^2;q) = -q^3 (q)_{\infty} (q^2; q^2)_{\infty}$ in order to obtain
\begin{equation} \label{sss1}
\mathcal{V}_{1}^{(1)}(x;q) = -\frac{1}{1+x} \left( \frac{1}{\Theta(-q;q^2)} f_{1,1,2}(-q,-q;q^2) - \frac{1}{\Theta(qx;q^2)} f_{1,1,2}(x^{-1} q, -q; q^2) \right).
\end{equation}
By (\ref{ftofalse}) and some simplifications, we have
\begin{equation} \label{sss2}
\begin{aligned}
f_{1,1,2}(-q,-q;q^2) & = \frac{\Theta(-q;q^2)}{2} \sum_{r \in \mathbb{Z}} \text{sg}(r) q^{r^2 - r} + \frac{\Theta(-q;q^4)}{2} \sum_{r \in \mathbb{Z}} \text{sg}(r) q^{\binom{r+1}{2}} \\
& = \Theta(-q;q^2)
\end{aligned}
\end{equation}
where in (\ref{sss2}) the first sum is 2 while the second sum is 0 after letting $r \to -r-1$ and using that $\text{sg}(-r-1) = - \text{sg}(r)$. Similarly, one can check that
\begin{equation} \label{sss3}
f_{1,1,2}(x^{-1} q, -q; q^2) = \frac{\Theta(qx;q^2)}{2} \sum_{r \in \mathbb{Z}} \text{sg}(r) (-1)^r x^r q^{r^2 - r} + \frac{\Theta(-q;q^4)}{2} \sum_{r \in \mathbb{Z}} \text{sg}(r) (-1)^r x^{r+1} q^{\binom{r+1}{2}}.
\end{equation}
By (\ref{sgo}), (\ref{sss1})--(\ref{sss3}) and rearranging terms, we have
\begin{equation} \label{sss4}
\begin{aligned}
\mathcal{V}_{1}^{(1)}(x;q) & = -\frac{1}{1+x} \Biggl( 1 - \sum_{r \geq 0} (-1)^r x^r q^{r^2 - r} + \frac{1}{2} \sum_{r \in \mathbb{Z}} (-1)^r x^r q^{r^2 - r} \\
& - \frac{\Theta(-q;q^4)}{\Theta(qx;q^2)} \sum_{r \geq 0} (-1)^r x^{r+1} q^{\binom{r+1}{2}} + \frac{\Theta(-q;q^4)}{2 \Theta(qx;q^2)}  \sum_{r \in \mathbb{Z}} (-1)^r x^{r+1} q^{\binom{r+1}{2}} \Biggr).
\end{aligned}
\end{equation}
Finally, we apply (\ref{theta}) to the second and fourth sums, cancel the resulting theta functions, remove the $r=0$ term and then perform the shift $r \to r+1$ in the first sum on the right-hand side of (\ref{sss4}). In total, 
\begin{equation*}
\mathcal{V}_{1}^{(1)}(x;q) = -\frac{1}{1+x} \left( \sum_{r \geq 0} (-1)^{r+1} x^{r+1} q^{r(r+1)} - \frac{\Theta(-q;q^4)}{2 \Theta(qx;q^2)}\sum_{r \geq 0} (-1)^r x^{r+1} q^{\binom{r+1}{2}} \right)
\end{equation*}
which is (\ref{pt3}) after simplification. 

\subsection{Recovering (\ref{pt4})}  We first set $t=m=1$ in Theorem \ref{main4}, then use \cite[Eq. (1.7)]{hm} and (\ref{theta1}) to deduce that $f_{1,2,1}(q^{-2}, q^4; q^2)=0$, $f_{1,2,1}(1,q^4; q^2) = q^2 (q^2; q^2)_{\infty}^2$ and $f_{1,2,1}(q^2, q^4; q^2)=(q^2; q^2)_{\infty}^2$ in order to obtain
\begin{equation} \label{ssss1}
\mathcal{O}_{1}^{(1)}(x;q) = \frac{q}{\Theta(xq;q^2)} \Bigl( f_{1,3,12}(x^{-1} q^3, -q^{20}; q^2) - q^4 f_{1,3,12}(x^{-1} q^5, -q^{28}; q^2) \Bigr)
\end{equation}
By (\ref{theta}), (\ref{ftofalse}), (\ref{theta1}) and (\ref{qpi}) and some simplifications, we have
\begin{equation} \label{ssss2}
\begin{aligned}
& f_{1,3,12}(x^{-1} q^3, -q^{20}; q^2) - q^4 f_{1,3,12}(x^{-1} q^5, -q^{28}; q^2) \\
& \qquad \qquad = -\frac{q^{-1} \Theta(xq;q^2)}{2} \sum_{r \in \mathbb{Z}} \text{sg}(r) (-1)^r x^{3r+1} q^{3r^2 + 2r} (1 + xq^{2r+1}) \\
& \qquad \qquad + \frac{1}{2} \sum_{t=0}^{11} (-1)^t x^{-t} q^{t^2 + 2t} \frac{\Theta(q^{2t+4}; q^8) \Theta(q^{16 + 4t}; q^{16})}{(q^{16}; q^{16})_{\infty}} \sum_{r \in \mathbb{Z}} \text{sg}(r) x^{-12r} q^{6tr + 36r^2}.
\end{aligned}
\end{equation}
For $t$ even, one of the theta functions in the numerator of the second term on the right-hand side of (\ref{ssss2}) is zero. Thus, it suffices to consider $t$ odd. There are two steps. We first replace $t$ with $4t+1$ and let $0 \leq t \leq 2$ in the second term on the right-hand side of (\ref{ssss2}). This eventually yields
\begin{equation} \label{ssss3}
\frac{q^{-1} x^{-1} (q^2; q^2)_{\infty}}{2} \sum_{k \in \mathbb{Z}} \text{sg}(k) x^{-4k} q^{2 \binom{2k+1}{2}}. 
\end{equation}
Next, we replace $t$ with $4t+3$ and let $0 \leq t \leq 2$ in the third term on the right-hand side of (\ref{ssss2}). This gives
\begin{equation} \label{ssss4}
- \frac{q^{-1} x^{-1} (q^2; q^2)_{\infty}}{2} \sum_{k \in \mathbb{Z}} \text{sg}(k) x^{-4k-2} q^{2 \binom{2k+2}{2}}.
\end{equation}
We now combine (\ref{ssss3}) and (\ref{ssss4}), perform the shift $k \to -k-1$ and use that $\text{sg}(-k-1) = - \text{sg}(k)$ to obtain
\begin{equation} \label{ssss5}
\frac{q^{-1} (q^2; q^2)_{\infty}}{2} \sum_{k \in \mathbb{Z}} \text{sg}(k) (-1)^k x^{2k+1} q^{2 \binom{k+1}{2}}.
\end{equation}
Now, after inserting (\ref{ssss5}) into (\ref{ssss2}), (\ref{ssss1}), using (\ref{sgo}) and rearranging terms, we have
\begin{equation} \label{ssss6}
\begin{aligned}
\mathcal{O}_{1}^{(1)}(x;q) & = -\frac{1}{2} \Biggl( 2 \sum_{r \geq 0} (-1)^r x^{3r+1} q^{3r^2 + 2r} (1 + xq^{2r+1}) - \sum_{r \in \mathbb{Z}} (-1)^r x^{3r+1} q^{3r^2 + 2r} (1 + xq^{2r+1}) \Biggr) \\
&+ \frac{(q^2; q^2)_{\infty}}{2\Theta(xq; q^2)} \Biggl( 2 \sum_{k \geq 0} (-1)^k x^{2k+1} q^{2 \binom{k+1}{2}} - \sum_{k \in \mathbb{Z}} (-1)^k x^{2k+1} q^{2 \binom{k+1}{2}} \Biggr).
\end{aligned}
\end{equation}
Note that two applications of (\ref{theta}) followed by (\ref{qpi}) implies
\begin{equation} \label{ssss7}
\sum_{r \in \mathbb{Z}} (-1)^r x^{3r+1} q^{3r^2 + 2r} (1 + xq^{2r+1}) = x \frac{\Theta(-xq; q^2) \Theta(x^2 q^4; q^4)}{(q^4; q^4)_{\infty}}
\end{equation}
while (\ref{theta}) yields
\begin{equation} \label{ssss8}
\sum_{k \in \mathbb{Z}} (-1)^k x^{2k+1} q^{2\binom{k+1)}{2}} = x \Theta(x^2 q^2; q^2).
\end{equation}
Thus, (\ref{pt4}) follows from (\ref{ssss6})--(\ref{ssss8}), cancelling the theta functions and then replacing $x$ with $-x$.

\section{Concluding Remarks}
First, Kim and Lovejoy \cite{kl2} also consider unimodal sequences where $\sum b_i$ is a partition into parts at most $c-k$ where $k$ is the size of the Durfee square of the partition $\sum a_i$. Let $V(m,n)$ denote the number of such sequences of weight $n$ and rank $m$ and consider its generating function \cite[Eq. (3.1)]{kl2}
\begin{equation*}
V(x;q) :=  \sum_{\substack{n \geq 0 \\ m \in \mathbb{Z}}} V(m,n) x^m q^n = \sum_{n \geq 0} \frac{(q^{n+1})_n q^n}{(xq)_n (q/x)_n}
\end{equation*}
which satisfies the partial theta function identity \cite{andrewsI}, \cite[Entry 6.6.1]{abII}
\begin{equation} \label{pt5}
V(x;q) = (1-x) \left( \sum_{n \geq 0} x^n q^{n(n+1)} + \frac{1}{(x)_{\infty} (q/x)_{\infty}} \sum_{n \geq 0} x^{3n+1} q^{n(3n+2)} (1 - xq^{2n+1}) \right)
\end{equation}
for $x \neq 0$ and the Hecke-Appell type expansion \cite[Eq. (3.3)]{kl2}
\begin{equation} \label{base4}
V(x;q) = \frac{(1-x)}{(q)_{\infty}^2} \left( \sum_{r,s \geq 0} - \sum_{r, s < 0} \right) \frac{(-1)^r q^{\binom{r}{2} + 3rs + 6 \binom{s}{2} + 2r + 5s} (1 - q^{r+2s+1})}{1-xq^r}. 
\end{equation}
For $t$, $m \in \mathbb{Z}$ with $t \geq 1$, $0 \leq m \leq 3t-1$, consider the generalization
\begin{equation} \label{gen3}
\begin{aligned}
& {\ell}_{t,m}(x;q) \\
& := \left( \sum_{r,s \geq 0} - \sum_{r, s < 0} \right) \frac{(-1)^r q^{\binom{r}{2} + 3trs + 3t(3t-1)\binom{s}{2} + (m+1)r + (\binom{3t}{2} + 3t-1)s} (1 - q^{(3t-2m)r + 2 \binom{3t-1}{2}s + \binom{3t-1}{2}})}{1-xq^r}
\end{aligned}
\end{equation}
and
\begin{equation} \label{realgen3}
V_{t}^{(m)}(x;q) := \frac{(1-x)}{(q)_{\infty}^2} {\ell}_{t,m}(x;q).
\end{equation}
By (\ref{base4})--(\ref{realgen3}), $V_{1}^{(1)}(x;q) = V(x;q)$. A similar argument as in the proofs of Theorems \ref{main1}--\ref{main4} shows that 
$V_{t}^{(m)}(x;q)$ is a sum of mixed false theta series, the triple sums appearing in \cite{morttriple} and a modular form. Unfortunately, it is not clear how the $t=m=1$ case of this result recovers (\ref{pt5}). Second, it would be worthwhile to investigate the similarities between Theorems \ref{main1} and \ref{main4} and Warnaar's generalizations of (\ref{pt1}) and (\ref{pt4}) (see \cite[Theorem 1.4, Theorem 7.1]{warnaar}). Finally, given that false theta functions are examples of quantum modular forms \cite[Section 4.4]{go}, it would highly desirable to determine whether Theorems \ref{main1}--\ref{main4} lead to the construction of new families of quantum Jacobi forms in the spirit of \cite{folsom}. 

\section*{Acknowledgements} 
The first author would like to acknowledge funding received from the UCD School of Mathematics and Statistics Research Demonstratorship. The second author would like to thank the National Institute of Science Education and Research, Bhubaneswar, India for their hospitality and support during his visit from January 5-12, 2024. The second author was partially funded by the Irish Research Council Advanced Laureate Award IRCLA/2023/1934. Finally, the authors would like to thank Jeremy Lovejoy, Martin Rubey and Matthias Storzer for their helpful comments and suggestions and to expresses our sincere gratitude to the referee for the calculations in Section 4 and for advice which greatly improved the paper.

\end{document}